\documentclass[11pt]{article}

\usepackage{mathrsfs,amssymb,amsmath,amsfonts}
\usepackage{amsthm}
\usepackage{bbm}
\usepackage{tipa}
\usepackage{enumerate}
\usepackage{color}
\usepackage[titletoc,title]{appendix}
\usepackage{cite}
\usepackage[hyphens]{url}
\usepackage[bookmarks=false,hidelinks]{hyperref}

\usepackage{arydshln}% For a complax matrix

\newtheorem{theorem}{Theorem}[section]
\newtheorem{lemma}[theorem]{Lemma}
\newtheorem{definition}[theorem]{Definition}
\newtheorem{proposition}[theorem]{Proposition}
\newtheorem{corollary}[theorem]{Corollary}
\newtheorem{remark}[theorem]{Remark}

\theoremstyle{definition}
\newtheorem{example}[theorem]{Example}

\newenvironment{keywords}{{\bf Key words: }}{}
\newenvironment{AMS}{{\bf AMS subject classification: }}{}

\newcommand{\esssup}{\mathop{\mathrm{esssup}}}

\numberwithin{equation}{section}

\newcommand{\udots}{\mathinner{\mskip1mu\raise1pt\vbox{\kern7pt\hbox{.}}
\mskip2mu\raise4pt\hbox{.}\mskip2mu\raise7pt\hbox{.}\mskip1mu}}

\begin{document}

\title{Exact Controllability for Mean-Field Type Linear Game-Based Control Systems
\footnote{This work is supported in part by the National Natural Science
Foundation of China (12271304).}}

\author{Cui Chen and Zhiyong Yu \footnote{School of Mathematics, Shandong University, Jinan 250100, China. E-mails: {\tt chencuiyde@mail.sdu.edu.cn} (C. Chen), {\tt yuzhiyong@sdu.edu.cn} (Z. Yu). The corresponding author is Zhiyong Yu.}  }

\maketitle

\begin{abstract}
Motivated by the self-pursuit of controlled objects, we consider the exact controllability of a linear mean-field type game-based control system (MF-GBCS, for short) generated by a linear-quadratic (LQ, for short) Nash game. A Gram-type criterion for the general time-varying coefficients case and a Kalman-type criterion for the special time-invariant coefficients case are obtained. At the same time, the equivalence between the exact controllability of this MF-GBCS and the exact observability of a dual system is established. Moreover, an admissible control that can steer the state from any initial vector to any terminal random variable is constructed in closed form. 
\end{abstract}

\\


\section{Introduction}

Recently, Zhang and Guo \cite{ZhG-19AC, ZhG-19SICON, ZhG-21SICON} proposed a new kind of control systems, named the game-based control systems (GBCSs, for short). This kind of systems go beyond the basic assumption in the traditional control theory that the controlled objects have no intelligence and self-pursuit, therefore GBCSs can be effectively used to model various controlled intelligent systems in the fields of economy, finance, life science, artificial intelligence and so on (see \cite{ChWW-15, LCGCB-19, LMYG-17, MDJ-18, MS-15, Pin-77, vEP-02, ZhM-13} for example). Specifically, the controlled objects in this situation, such as humans, creatures, intelligent devices, etc., are called {\it agents}. The self-pursuit of each agent is shown in the mathematical model as trying to minimize her/his own cost functional (or maximize a certain payoff). Then, the agents form a non-cooperative Nash-type game. Besides, there is also a controller which is called a {\it regulator} in this paper. Due to the self-pursuit of the agents, the controlled system faced by the regulator is essentially changed from the original one used to describe the evolution of things to the Hamiltonian system generated by agents' game. After removing the initial condition, this controlled Hamiltonian system is called a GBCS. Facing GBCSs, the control aims of the regulator includes optimality (see \cite{MX-15, Mu-18, WWZh-20} for example), controllability (see \cite{LY-22+, ZhG-19AC, ZhG-19SICON}), stability (see \cite{ZhG-21SICON}) and so on. In the present paper, we consider the exact controllability of a linear MF-GBCS, which can be regarded as a continuation of \cite{LY-22+, ZhG-19AC, ZhG-19SICON}.

The research on mean-field type stochastic controls and mean-field type stochastic differential games led by Huang et al. \cite{HMC-06} and Lasry and Lions \cite{LL-07} has gradually reached an upsurge since 2006 and has continued to this day. Specifically, the mean-field items in this paper refer to the mathematical expectation items appearing in the state equation (see \eqref{2:Sys}) and the cost functionals (see \eqref{2:Cost}). On the one hand, as the state equation, this kind of stochastic differential equations (SDEs, for short) with mathematical expectation items are often called McKean-Vlasov SDEs or mean-field type SDEs (MF-SDEs, for short). A remarkable feature of MF-SDEs is that they are widely used to effectively describe the dynamic evolution of particle systems or large population systems with mean-field interactions. On the other hand, the mean-field items in the cost functionals can effectively characterize the subjective attitude of agents towards risk. For example, the mean-field items appearing in the cost functionals of Example \ref{2.3:Example} are in the form of variances:
\[
\bar\alpha_i \mbox{Var}\ [x(T)] = \bar\alpha_i \mathbb E\Big[ \big| x(T) -\mathbb E[x(T)] \big|^2 \Big], \quad
\bar\beta_i \int_0^T \mbox{Var}\ [x(s)]\, \mathrm ds = \bar\beta_i \mathbb E\int_0^T \big| x(s) -\mathbb E[x(s)] \big|^2\, \mathrm ds
\]
($i=1,2$). When $\bar\alpha_i>0$ and $\bar\beta_i>0$, the above variances express Agent $i$'s risk aversion that she/he hopes the state is not too sensitive to the changes of random scenarios. On the contrary, when $\bar\alpha_i<0$ and $\bar\beta_i<0$, the above variances express Agent $i$'s risk seeking, i.e. the psychology of risk takers. Compared with the abundant research on the optimality of mean-field systems in the literature (see \cite{BDL-11, LLY-20, PW-18, Yo-13} for example), the research on the controllability is extremely rare (see \cite{Gor-14, Y-21}). The research on the exact controllability of the linear MF-GBCS in the present paper will enrich the direction of controllability.

Controllability is one of the core research issues in the control theory. Relatively speaking, the research on the controllability of deterministic systems is mature and fruitful. For surveys of ordinary differential equation (ODE, for short) systems and partial differential equation (PDE, for short) systems, one can refer to Lee and Markus \cite{LM-67} and Coron \cite{Coron-07} respectively. However, the study of controllability of stochastic systems is rare. One important reason for this phenomenon is as follows. The terminal state in a deterministic system is located in a finite dimensional vector space, namely $x_T \in \mathbb R^n$. However, for a stochastic system, the terminal state is located in an infinite dimensional space of random variables, such as $x_T \in L^2_{\mathcal F_T}(\Omega;\mathbb R^n)$ (the rigorous definition will be given in Section \ref{SEC:Form}). Then, in the research of controllability of stochastic systems, there is an essential difficulty caused by infinite dimension of space. In 1994, with the help of the theory of backward stochastic differential equations (BSDEs, for short, see \cite{PP-90}), Peng \cite{P-94} overcame this difficulty for the first time and obtained a Kalman-type criterion for the exact controllability of linear SDE systems with time-invariant coefficients. This result was improved by L\"u and Zhang \cite{LZh-BOOK} from Peng's Kalman matrix with infinite number of columns to a sub-matrix with finite number of columns. Along this line, we continue to mention several results which are closely related to the present paper and are at the core of the exact controllability theory. In 2010, Liu and Peng \cite{LP-10} obtained a Gram-type criterion for the exact controllability of linear SDE systems with time-varying deterministic coefficients. This result was improved by Wang et al. \cite{WYYY-17} to the general case of time-varying random coefficients. In addition, a dual relationship between the exact controllability of an SDE system and the exact observability of its dual system was rigorously established in \cite{WYYY-17}. Moreover, these results of SDE systems were successfully extended to the new kind of linear stochastic GBCSs by Zhang and Guo \cite{ZhG-19SICON} and Liu and Yu \cite{LY-22+}. With the help of the theory of mean-field type BSDEs (MF-BSDEs, for short, see \cite{BDLP-09} for example) and some delicate analysis of mean-field dependence, Yu \cite{Y-21} obtained a Gram-type criterion and a Kalman-type criterion for the exact controllability of linear MF-SDE systems. The present paper can be regarded as the development and enrichment of \cite{Y-21} in MF-GBCSs.

The novelties and contributions of this paper are summarized as follows.
\begin{enumerate}[(i).]
\item Comparing with the previous literature on GBCSs (see \cite{LY-22+, ZhG-19AC, ZhG-19SICON, ZhG-21SICON} for example), we introduce a concept of {\it admissible strategy} (see Definition \ref{2.1D:Str}) to more accurately characterize the hierarchical relationship between the regulator and the agents, that is, the selection of the admissible controls of agents changes with the different admissible control announced by the regulator.

\item For the core difficulties commonly existing in stochastic differential games: the existence and uniqueness of Nash equilibria, firstly, we use some delicate variational analysis to equate them to the existence and uniqueness of a mean-field type forward-backward SDE (MF-FBSDE, for short, see \eqref{2:MF-FBSDE}). Secondly, we adopt a definition of exact controllability of the MF-GBCS (see Definition \ref{2.2:Def}) which is weaker than that in \cite{ZhG-19SICON} but is similar to that in \cite{LY-22+}. On the one hand, we avoid the uniqueness of MF-FBSDE \eqref{2:MF-FBSDE} which is more challenging than the existence, so that the definition of weak form and the conclusions obtained have a wider range of applications. On the other hand, the existence of MF-FBSDE \eqref{2:MF-FBSDE} is implied as a natural deduction of exact controllability, which provides an alternative way for the study of Nash equilibria (see Remark \ref{2.2:Rem}).

\item For the difficulty that the dimension of terminal state space is infinite, we combine and extend the idea and method in \cite{P-94, LY-22+, WYYY-17, Yo-13, Y-21} to the MF-GBCSs we care about (see \eqref{2:MF-GBCS}). Under some suitable assumptions (see Assumption (H4) in Section \ref{SEC:Exact}), we equivalently transform the exact controllability of MF-GBCS \eqref{2:MF-GBCS} to the {\it exact null-controllability} of a backward system (see \eqref{3:BS}). In other words, the terminal state $x_T$ degenerates from any point in the infinite dimensional space $L^2_{\mathcal F_T}(\Omega;\mathbb R^n)$ that needs be controlled to the zero point of the space $L^2_{\mathcal F_T}(\Omega;\mathbb R^n)$. This makes the difficulty of infinite dimension overcome.

\item For the exact controllability of MF-GBCS \eqref{2:MF-GBCS}, we obtain sufficient and necessary Gram-type criterion (see Theorem \ref{Gram-Cri}) and Kalman-type criterion (see Theorem \ref{Kalman-Cri}). Moreover, we establish the equivalence between the exact controllability of MF-GBCS \eqref{2:MF-GBCS} and the {\it exact observability} of a dual system (see System \eqref{3.4:DS} and \eqref{3.4:K*} and Theorem \ref{3.4T:ECEO}), which provides an indirect research approach of controllability through observability. Furthermore, we construct an admissible control in closed form for the regulator that can steer the state process $x(\cdot)$ from any existing initial value $x_0\in \mathbb R^n$ to any desired terminal target $x_T \in L^2_{\mathcal F_T}(\Omega; \mathbb R^n)$ (see \eqref{5:Su0*} and Theorem \ref{4:THM}).
\end{enumerate}

The rest of this paper is organized as follows. In Section \ref{SEC:Form},  we establish the equivalence between the existence and uniqueness of agents' LQ Nash equilibria and that of linear MF-FBSDE \eqref{2:MF-FBSDE}. Then, the formulation of the exact controllability of MF-GBCS \eqref{2:MF-GBCS} faced by the regulator is given. In Section \ref{SEC:Exact}, we establish the equivalence between the exact controllability of MF-GBCS \eqref{2:MF-GBCS} and the exact null-controllability of the backward system \eqref{3:BS}. Then, for the exact controllability of MF-GBCS \eqref{2:MF-GBCS}, a Gram-type criterion for the general time-varying coefficients case and a Kalman-type criterion for the special time-invariant coefficients case are obtained. An indirect approach to judge the exact controllability through the exact observability of the dual system \eqref{3.4:DS} and \eqref{3.4:K*} is also provided. In Section \ref{SEC:Steer}, an admissible control of the regulator that can steer the state $x(\cdot)$ from any initial value $x_0$ to any terminal value $x_T$ is constructed. In Section \ref{SEC:Sol}, we summarize the results of Sections \ref{SEC:Form}, \ref{SEC:Exact} and \ref{SEC:Steer} to give a complete solution to the two-layer problem consisting of exact controllability and Nash game.

\section{Problem formulation}\label{SEC:Form}

Let $\mathbb R^n$ be the $n$-dimensional Euclidean space equipped with the Euclidean inner product $\langle \cdot,\ \cdot \rangle$ and the induced norm $|\cdot|$. Let $\mathbb R^{m\times n}$ be the collection of all $(m\times n)$ matrices and $\mathbb S^n$ consist of all $(n\times n)$ symmetrical matrices. Clearly, both $\mathbb R^{m\times n}$ and $\mathbb S^n$ are Euclidean spaces.

Let $(\Omega, \mathcal F, \mathbb F, \mathbb P)$ be a complete filtered probability space on which is defined a one-dimensional Brownian motion $W(\cdot)$. In this paper, the dimension of Brownian motion is set to be $d=1$ for the convenience of notation. The case of $d>1$ can be studied similarly. Let $T>0$ be a fixed time horizon, $\mathbb F =\{ \mathcal F_t,\ 0\leq t\leq T \}$ be the natural filtration generated by $W(\cdot)$ and augmented by all $\mathbb P$-null sets, and $\mathcal F =\mathcal F_T$.

We continue to introduce some Banach (sometimes more accurately, Hilbert) spaces consisting of random variables, deterministic functions or stochastic processes which will be used in this paper. 

\begin{itemize}
\item $L^2_{\mathcal F_T}(\Omega;\mathbb R^n)$ is the set of $\mathcal F_T$-measurable random variables $\varphi:\Omega\rightarrow \mathbb R^n$ such that
\[
\Vert \varphi \Vert_{L^2_{\mathcal F_T}(\Omega;\mathbb R^n)} := \Big\{ \mathbb E \big[ |\varphi|^2 \big] \Big\}^{1/2} <\infty.
\]

\item $L^\infty(0,T;\mathbb R^n)$ is the set of Lebesgue measurable functions $\varphi:[0,T] \rightarrow \mathbb R^n$ such that
\[
\Vert \varphi(\cdot) \Vert_{L^\infty(0,T;\mathbb R^n)} := \esssup_{s\in [0,T]} |\varphi(s)| <\infty.
\]

\item $L^2_{\mathbb F}(0,T;\mathbb R^n)$ is the set of $\mathbb F$-progressively measurable stochastic processes $\varphi: [0,T] \times \Omega \rightarrow \mathbb R^n$ such that
\[
\Vert \varphi(\cdot) \Vert_{L^2_{\mathbb F}(0,T;\mathbb R^n)} := \bigg\{ \mathbb E\int_0^T |\varphi(s)|^2\, \mathrm ds\bigg\}^{1/2} <\infty.
\]

\item $L^2_{\mathbb F}(\Omega;C(0,T;\mathbb R^n))$ is the set of $\mathbb F$-progressively measurable stochastic processes $\varphi: [0,T] \times \Omega \rightarrow \mathbb R^n$ such that for almost all $\omega \in \Omega$, $s \mapsto \varphi(s,\omega)$ is continuous, and
\[
\Vert \varphi(\cdot) \Vert_{L^2_{\mathbb F}(\Omega;C(0,T;\mathbb R^n))} := \bigg\{ \mathbb E\bigg[ \sup_{s\in [0,T]} |\varphi(s)|^2 \bigg] \bigg\}^{1/2} <\infty.
\]
\end{itemize}

In this paper, we are interested in a mean-field type two-layer controlled system, in which the lower layer is composed of $N$ ($\geq 1$) agents, while the upper layer has only one regulator. We assume that the dynamic of the system is given by the following linear MF-SDE:
\begin{equation}\label{2:Sys}
\left\{
\begin{aligned}
& \mathrm dx(s) = \Big\{ A(s)x(s) +\bar A(s) \mathbb E[x(s)] +B(s)u(s) +\bar B(s) \mathbb E[u(s)] \Big\}\, \mathrm ds\\
& \quad +\Big\{ C(s)x(s) +\bar C(s) \mathbb E[x(s)] +D(s)u(s) +\bar D(s) \mathbb E[u(s)] \Big\}\, \mathrm dW(s),\quad s\in [0,T],\\
& x(0) =x_0,
\end{aligned}
\right.
\end{equation}
where $A(\cdot)$, $\bar A(\cdot)$, $C(\cdot)$, $\bar C(\cdot) \in L^\infty (0,T;\mathbb R^{n\times n})$, $B(\cdot)$, $\bar B(\cdot)$, $D(\cdot)$, $\bar D(\cdot) \in L^\infty(0,T;\mathbb R^{n\times m})$ and $x_0 \in \mathbb R^n$. Here, $u(\cdot)\in L^2_{\mathbb F}(0,T;\mathbb R^m)$ is a concise representation of the admissible control processes of the regulator and the agents. In fact, it admits a decomposition:
\begin{equation}\label{2:N1}
u(\cdot) = \begin{pmatrix} u_0(\cdot)^\top & u_1(\cdot)^\top & \cdots & u_N(\cdot)^\top \end{pmatrix}^\top =: \begin{pmatrix} u_0(\cdot)^\top & u_{-0}(\cdot)^\top \end{pmatrix}^\top,
\end{equation}
where $u_0(\cdot) \in L^2_{\mathbb F}(0,T;\mathbb R^{m_0})$ is the admissible control of the regulator and $u_i(\cdot) \in L^2_{\mathbb F}(0,T;\mathbb R^{m_i})$ is the admissible control of Agent $i$ ($i=1,2,\dots,N$). Here and hereafter, the superscript ``$\top$'' denotes the transpose of a vector or a matrix. Obviously, $m = \sum_{i=0}^N m_i =: m_0 +m_{-0}$. For convenience of later use, the notation $u_{-0}(\cdot) \in L^2_{\mathbb F}(0,T;\mathbb R^{m_{-0}})$ is also introduced to denote an $N$-tuple of admissible controls of all agents (see \eqref{2:N1}). Moreover, the related coefficients of MF-SDE \eqref{2:Sys} admit the  corresponding decompositions:
\begin{equation}\label{2:N2}
\begin{aligned}
& B(\cdot) = \begin{pmatrix} B_0(\cdot) & B_1(\cdot) & \cdots & B_N(\cdot)  \end{pmatrix} =: \begin{pmatrix} B_0(\cdot) & B_{-0}(\cdot) \end{pmatrix},\\
& \bar B(\cdot) = \begin{pmatrix} \bar B_0(\cdot) & \bar B_1(\cdot) & \cdots & \bar B_N(\cdot)  \end{pmatrix} =: \begin{pmatrix} \bar B_0(\cdot) & \bar B_{-0}(\cdot) \end{pmatrix},\\
& D(\cdot) = \begin{pmatrix} D_0(\cdot) & D_1(\cdot) & \cdots & D_N(\cdot)  \end{pmatrix} =: \begin{pmatrix} D_0(\cdot) & D_{-0}(\cdot) \end{pmatrix},\\
& \bar D(\cdot) = \begin{pmatrix} \bar D_0(\cdot) & \bar D_1(\cdot) & \cdots & \bar D_N(\cdot)  \end{pmatrix} =: \begin{pmatrix} \bar D_0(\cdot) & \bar D_{-0}(\cdot) \end{pmatrix},
\end{aligned}
\end{equation}
where $B_i(\cdot)$, $\bar B_i(\cdot)$, $D_i(\cdot)$, $\bar D_i(\cdot) \in L^\infty(0,T;\mathbb R^{n\times m_i})$ ($i=1,2,\dots,N$). By the theory of MF-SDEs (see Yong \cite[Proposition 2.6]{Yo-13} for example), for any $x_0\in \mathbb R^n$ and any $u(\cdot) \in L^2_{\mathbb F}(0,T;\mathbb R^m)$, \eqref{2:Sys} admits a unique solution $x(\cdot) \in L^2_{\mathbb F}(\Omega;C(0,T;\mathbb R^n))$ which is called the state process at $(x_0, u(\cdot))$.

\subsection{LQ Nash game at the lower layer}

In our two-layer model, the agents at the lower layer always play an LQ Nash-type nonzero-sum stochastic differential game after an admissible control of the regulator at the upper layer is given. In detail, for any given $(x_0, u_0(\cdot)) \in \mathbb R^n \times L^2_{\mathbb F}(0,T;\mathbb R^{m_0})$, Agent $i$ is set to minimize the following quadratic cost functional:
\begin{equation}\label{2:Cost}
\begin{aligned}
& J_i\big( u_{-0}(\cdot); x_0, u_0(\cdot) \big) = \mathbb E \bigg\{ \big\langle H_ix(T),\ x(T) \big\rangle +\int_0^T \Big[ \big\langle Q_i(s)x(s),\ x(s) \big\rangle +\big\langle R_i(s)u_i(s),\ u_i(s) \big\rangle \Big]\, \mathrm ds \bigg\}\\
& \qquad +\big\langle \bar H_i \mathbb E [x(T)],\ \mathbb E[x(T)] \big\rangle +\int_0^T \Big[ \big\langle \bar Q_i(s) \mathbb E[x(s)],\ \mathbb E[x(s)] \big\rangle +\big\langle \bar R_i(s)\mathbb E[u_i(s)],\ \mathbb E[u_i(s)] \big\rangle \Big]\, \mathrm ds,
\end{aligned}
\end{equation}
where $H_i$, $\bar H_i \in \mathbb S^n$, $Q_i(\cdot)$, $\bar Q_i(\cdot) \in L^\infty (0,T;\mathbb S^n)$ and $R_i(\cdot)$, $\bar R_i(\cdot) \in L^\infty(0,T;\mathbb S^{m_i})$, by choosing her/his admissible control $u_i(\cdot) \in L^2_{\mathbb F}(0,T;\mathbb R^{m_i})$ ($i=1,2,\dots,N$). Clearly, the cost functionals \eqref{2:Cost} are well-defined, i.e., for any $(x_0, u_0(\cdot))\in \mathbb R^n \times L^2_{\mathbb F}(0,T;\mathbb R^{m_0})$, any $u_{-0}(\cdot) \in L^2_{\mathbb F}(0,T;\mathbb R^{m_{-0}})$ and any $i=1,2,\dots, N$, we have $-\infty < J_i(u_{-0}(\cdot); x_0, u_0(\cdot)) <+\infty$.

In general, the selection of agents' admissible controls will be reasonably changed according to the different regulator's admissible control given in advance. We introduce the following concept of admissible strategy to emphasize this selection dependency.

\begin{definition}\label{2.1D:Str}
An $N$-tuple of {\rm admissible strategies} of agents is a mapping 
\[
\mathbbm u_{-0} = \begin{pmatrix} \mathbbm u_1^\top & \mathbbm u_2^\top & \cdots & \mathbbm u_N^\top \end{pmatrix}^\top : \mathbb R^n \times L^2_{\mathbb F}(0,T;\mathbb R^{m_0}) \rightarrow L^2_{\mathbb F}(0,T;\mathbb R^{m_{-0}}).
\]
All admissible strategies of agents are collected in $\mathscr U_{-0}$.
\end{definition}

Now, we formulate the LQ Nash game among the agents.

\medskip

\noindent {\bf Problem (LQG).} To find an $N$-tuple of admissible strategies $\mathbbm u_{-0}^*(\cdot) \in \mathscr U_{-0}$ of agents such that, for any $(x_0, u_0(\cdot)) \in \mathbb R^n \times L^2_{\mathbb F}(0,T;\mathbb R^{m_0})$,
\begin{equation}\label{2:Nash}
J_i \big( u^*_{-0}(\cdot); x_0, u_0(\cdot) \big) = \inf_{u_i(\cdot) \in L^2_{\mathbb F}(0,T;\mathbb R^{m_i})} J\big( u_i(\cdot), u^*_{-\{0,i\}}(\cdot); x_0,u_0(\cdot) \big),\quad i=1,2,\dots,N,
\end{equation}
where $u^*_{-0}(\cdot) = \mathbbm u_{-0}^*(x_0, u_0(\cdot))$ is the {\it outcome} of the strategy $\mathbbm u_{-0}^*(\cdot)$, and $u^*_{-\{0,i\}}(\cdot)$ represents all the components of $u^*_{-0}(\cdot)$ other than $u_i(\cdot)$. The admissible strategy satisfying \eqref{2:Nash} is called a {\it Nash equilibrium strategy} of the agents, and the outcome $u^*_{-0}(\cdot) = \mathbbm u_{-0}^*(x_0, u_0(\cdot))$ is called a {\it Nash equilibrium point} of the agents at $(x_0, u_0(\cdot))$.

\medskip

In order to solve Problem (LQG), we introduce a family of mean-field type Hamiltonian systems parameterized by $(x_0, u_0(\cdot)) \in \mathbb R^n \times L^2_{\mathbb F}(0,T;\mathbb R^{m_0})$ (the argument $s$ is suppressed for simplicity):
\begin{equation}\label{2:Ham}
\left\{
\begin{aligned}
& 0 = Ru_{-0}^* +\bar R\mathbb E[u_{-0}^*] +\widetilde B_{-0}^\top y_{-0} +\bar{\widetilde B}_{-0}^\top \mathbb E[y_{-0}] +\widetilde D_{-0}^\top z_{-0} +\bar{\widetilde D}_{-0}^\top \mathbb E[z_{-0}],\\
& \mathrm dx^* = \Big\{ Ax^* +\bar A\mathbb E[x^*] +B_{-0}u_{-0}^* +\bar B_{-0} \mathbb E[u_{-0}^*] +B_0 u_0 +\bar B_0 \mathbb E[u_0] \Big\}\, \mathrm ds\\
& \qquad +\Big\{ Cx^* +\bar C\mathbb E[x^*] +D_{-0}u_{-0}^* +\bar D_{-0} \mathbb E[u_{-0}^*] +D_0 u_0 +\bar D_0 \mathbb E[u_0] \Big\}\, \mathrm dW,\\
& \mathrm dy_{-0} = -\Big\{ \widetilde A^\top y_{-0} +\bar {\widetilde A}^\top \mathbb E[y_{-0}] +\widetilde C^\top z_{-0} +\bar{\widetilde C}^\top \mathbb E[z_{-0}] +Qx^* +\bar Q \mathbb E[x^*] \Big\}\, \mathrm ds +z_{-0}\, \mathrm dW,\\
& x^*(0) =x_0, \qquad y_{-0}(T) = Hx^*(T) +\bar H \mathbb E[x^*(T)],
\end{aligned}
\right.
\end{equation}
where we use the notations \eqref{2:N1}, \eqref{2:N2} and
\begin{equation}\label{2:N3}
\begin{aligned}
& \widetilde A(\cdot) = \mbox{diag } \big\{ \overbrace{ A(\cdot),\, A(\cdot),\, \dots,\, A(\cdot)}^{\mbox{the number is } N} \big\}, \  && \bar {\widetilde A}(\cdot) = \mbox{diag } \big\{ \overbrace{ \bar A(\cdot),\, \bar A(\cdot),\, \dots,\, \bar A(\cdot)}^{\mbox{the number is } N} \big\},\\
& \widetilde C(\cdot) = \mbox{diag } \big\{ \overbrace{ C(\cdot),\, C(\cdot),\, \dots,\, C(\cdot)}^{\mbox{the number is } N} \big\}, \  && \bar {\widetilde C}(\cdot) = \mbox{diag } \big\{ \overbrace{ \bar C(\cdot),\, \bar C(\cdot),\, \dots,\, \bar C(\cdot)}^{\mbox{the number is } N} \big\},\\
& \widetilde B_{-0}(\cdot) = \mbox{diag } \big\{ B_1(\cdot),\, B_2(\cdot),\, \dots,\, B_N(\cdot) \big\},\  && \bar {\widetilde B}_{-0}(\cdot) = \mbox{diag } \big\{ \bar B_1(\cdot),\, \bar B_2(\cdot),\, \dots,\, \bar B_N(\cdot) \big\},\\
& \widetilde D_{-0}(\cdot) = \mbox{diag } \big\{ D_1(\cdot),\, D_2(\cdot),\, \dots,\, D_N(\cdot) \big\},\  && \bar {\widetilde D}_{-0}(\cdot) = \mbox{diag } \big\{ \bar D_1(\cdot),\, \bar D_2(\cdot),\, \dots,\, \bar D_N(\cdot) \big\},\\
& R(\cdot) = \mbox{diag } \big\{ R_1(\cdot),\, R_2(\cdot),\, \dots,\, R_N(\cdot) \big\}, \ && \bar R(\cdot) = \mbox{diag } \big\{ \bar R_1(\cdot),\, \bar R_2(\cdot),\, \dots,\, \bar R_N(\cdot) \big\},\\
& Q(\cdot) = \begin{pmatrix} Q_1(\cdot)^\top & Q_2(\cdot)^\top & \cdots & Q_N(\cdot)^\top \end{pmatrix}^\top,\ && \bar Q(\cdot) = \begin{pmatrix} \bar Q_1(\cdot)^\top & \bar Q_2(\cdot)^\top & \cdots & \bar Q_N(\cdot)^\top \end{pmatrix}^\top,\\
& H = \begin{pmatrix} H_1^\top & H_2^\top & \cdots & H_N^\top \end{pmatrix}^\top, \ && \bar H = \begin{pmatrix} \bar H_1^\top & \bar H_2^\top & \cdots & \bar H_N^\top \end{pmatrix}^\top.
\end{aligned}
\end{equation}
We notice that there are four unknown processes $u_{-0}^*(\cdot)$, $x^*(\cdot)$, $y_{-0}(\cdot)$ and $z_{-0}(\cdot)$ in the Hamiltonian system \eqref{2:Ham}. For the convenience of later use, on the one hand, corresponding to \eqref{2:N3}, we introduce the decomposition:
\begin{equation}\label{2:N4}
y_{-0}(\cdot) = \begin{pmatrix} y_1(\cdot)^\top & y_2(\cdot)^\top & \cdots & y_N(\cdot)^\top \end{pmatrix}^\top,\quad z_{-0}(\cdot)= \begin{pmatrix} z_1(\cdot)^\top & z_2(\cdot)^\top & \cdots & z_N(\cdot)^\top \end{pmatrix}^\top, 
\end{equation}
where $y_i(\cdot)$ and $z_i(\cdot)$ ($i=1,2,\dots,N$) take values in $\mathbb R^n$; on the other hand, we introduce the following more concise notations: 
\begin{equation}\label{2:N5}
\begin{aligned}
& \theta(\cdot) = \begin{pmatrix} x^*(\cdot)^\top & y_{-0}(\cdot)^\top & z_{-0}(\cdot)^\top \end{pmatrix}^\top, \\
& \Theta = L^2_{\mathbb F}(\Omega; C(0,T;\mathbb R^n)) \times L^2_{\mathbb F}(\Omega; C(0,T;\mathbb R^{Nn})) \times L^2_{\mathbb F}(0,T;\mathbb R^{Nn}).
\end{aligned}
\end{equation}
As usual, a pair of processes $(u_{-0}^*(\cdot), \theta(\cdot)) \in L^2_{\mathbb F}(0,T;\mathbb R^{m_{-0}}) \times \Theta$ is called a solution to the Hamiltonian system \eqref{2:Ham} if it makes all the equations in \eqref{2:Ham} hold.

Next, we will link the solvability of Problem (LQG) with that of Hamiltonian system \eqref{2:Ham}. As usual, for a symmetrical matrix $A \in \mathbb S^n$, we denote $A\geq 0$ (resp. $>0$, $\leq 0$, $<0$) if $A$ is positive semi-definite (resp. positive definite, negative semi-definite, negative definite). For a function $A(\cdot): [0,T] \rightarrow \mathbb S^n$, we denote $A(\cdot) \geq 0$ (resp. $>0$, $\leq 0$, $<0$) if $A(s)\geq 0$ (resp. $>0$, $\leq 0$, $<0$) for almost all $s\in [0,T]$. Moreover, we denote $A(s)\gg 0$ (resp. $\ll 0$) if there exists a constant $\delta>0$ such that $A(\cdot) -\delta I_n \geq 0$ (resp. $A(\cdot) +\delta I_n \leq 0$). We introduce the following assumptions on the weight matrices of cost functionals which are standard in the LQ game theory:

\medskip

\noindent{\bf Assumption (H1).} $H_i$, $H_i +\bar H_i$, $Q_i(\cdot)$, $Q_i(\cdot) +\bar Q_i(\cdot) \geq 0$ and $R_i(\cdot)$, $R_i(\cdot) +\bar R_i(\cdot) \gg 0$ for any $i=1,2,\dots,N$.

\medskip

Now we are in the position to give the main result of this subsection:

\begin{theorem}
Let Assumption (H1) hold. Let $(x_0, u_0(\cdot)) \in \mathbb R^n \times L^2_{\mathbb F}(0,T;\mathbb R^{m_0})$ be given. Then, the existence and the uniqueness of Nash equilibrium point of Problem (LQG) at $(x_0, u_0(\cdot))$ are equivalent to that of solution of Hamiltonian system \eqref{2:Ham}, respectively. In this case, if $(u_{-0}^*(\cdot),\theta(\cdot)) \in L^2_{\mathbb F}(0,T;\mathbb R^{m_{-0}}) \times \Theta$ is a solution to \eqref{2:Ham}, then $u_{-0}^*(\cdot)$ provides a Nash equilibrium point of Problem (LQG) at $(x_0,u_0(\cdot))$.
\end{theorem}

\begin{proof}
For clarity, we divide the whole proof into two steps.

{\bf Step 1.} Let $(x_0, u_0(\cdot)) \in \mathbb R^n \times L^2_{\mathbb F}(0,T;\mathbb R^{m_0})$ be given. Let $u_{-0}^*(\cdot) \in L^2_{\mathbb F}(0,T;\mathbb R^{m_{-0}})$ be an $N$-tuple of admissible controls of agents. By the definition of Nash equilibrium point (see \eqref{2:Nash}), $u_{-0}^*(\cdot)$ is a Nash equilibrium point if and only if
\begin{equation}\label{2.1T:E1}
\begin{aligned}
J_p\big( u_p^*(\cdot) +\varepsilon u_p(\cdot), u_{-\{0,p\}}^*(\cdot); x_0, u_0(\cdot) \big) -J_p\big( u_p^*(\cdot), u_{-\{0,p\}}^*(\cdot); x_0, u_0(\cdot) \big) \geq 0,\\
\mbox{for any } \varepsilon \in \mathbb R, \mbox{ any } u_p(\cdot) \in L^2_{\mathbb F}(0,T;\mathbb R^{m_p}) \mbox{ and any } p=1,2,\dots,N.
\end{aligned}
\end{equation}

Next, we will introduce the so-called variational equations to simplify \eqref{2.1T:E1}. Firstly, we denote by $x^*(\cdot) \in L^2_{\mathbb F}(\Omega;C(0,T;\mathbb R^n))$  the state process at $(x_0, (u_0(\cdot)^\top , u_{-0}^*(\cdot)^\top)^\top)$, i.e., $x^*(\cdot)$ is the unique solution to the following MF-SDE:
\begin{equation}\label{2.1T:Sys}
\left\{
\begin{aligned}
& \mathrm dx^* = \Big\{ Ax^* +\bar A\mathbb E[x^*] +B_{-0}u_{-0}^* +\bar B_{-0} \mathbb E[u_{-0}^*] +B_0 u_0 +\bar B_0 \mathbb E[u_0] \Big\}\, \mathrm ds\\
& \qquad +\Big\{ Cx^* +\bar C\mathbb E[x^*] +D_{-0}u_{-0}^* +\bar D_{-0} \mathbb E[u_{-0}^*] +D_0 u_0 +\bar D_0 \mathbb E[u_0] \Big\}\, \mathrm dW,\\
& x^*(0) =x_0.
\end{aligned}
\right.
\end{equation}
For any $p=1,2,\dots,N$, we introduce a couple of notations:
\begin{equation}
\begin{aligned}
& \alpha_p =\sum_{\substack{1\leq i\leq N\\ i\neq p}} \Big( B_iu_i^* +\bar B_i \mathbb E[u_i^*] \Big) +B_0u_0 +\bar B_0\mathbb E[u_0], \\
& \beta_p =\sum_{\substack{1\leq i\leq N\\ i\neq p}} \Big( D_iu_i^* +\bar D_i \mathbb E[u_i^*] \Big) +D_0u_0 +\bar D_0\mathbb E[u_0],
\end{aligned}
\end{equation}
and rewrite \eqref{2.1T:Sys} as
\begin{equation}\label{2.1T:Sys1}
\left\{
\begin{aligned}
& \mathrm dx^* = \Big\{ Ax^* +\bar A\mathbb E[x^*] +B_p u_p^* +\bar B_p \mathbb E[u_p^*] +\alpha_p \Big\}\, \mathrm ds\\
& \qquad +\Big\{ Cx^* +\bar C\mathbb E[x^*] +D_p u_p^* +\bar D_p \mathbb E[u_p^*] +\beta_p \Big\}\, \mathrm dW,\\
& x^*(0) =x_0.
\end{aligned}
\right.
\end{equation}
Secondly, for any $\varepsilon\in \mathbb R$ and any $u_p(\cdot) \in L^2_{\mathbb F}(0,T;\mathbb R^{m_p})$, we denote by $u^\varepsilon_p(\cdot) = u_p^*(\cdot) +\varepsilon u_p(\cdot)$ and $x^\varepsilon(\cdot)\in L^2_{\mathbb F}(\Omega;C(0,T;\mathbb R^n))$ the unique solution to the following MF-SDE: 
\begin{equation}\label{2.1T:Sys2}
\left\{
\begin{aligned}
& \mathrm dx^\varepsilon = \Big\{ Ax^\varepsilon +\bar A\mathbb E[x^\varepsilon] +B_p u_p^\varepsilon +\bar B_p \mathbb E[u_p^\varepsilon] +\alpha_p \Big\}\, \mathrm ds\\
& \qquad +\Big\{ Cx^\varepsilon +\bar C\mathbb E[x^\varepsilon] +D_p u_p^\varepsilon +\bar D_p \mathbb E[u_p^\varepsilon] +\beta_p \Big\}\, \mathrm dW,\\
& x^\varepsilon(0) =x_0.
\end{aligned}
\right.
\end{equation}
Thirdly, similar to the relationship $u_p(\cdot) = (u^\varepsilon_p(\cdot) -u^*_p(\cdot))/ \varepsilon$, we denote $x^1(\cdot) =(x^\varepsilon(\cdot) -x^*(\cdot))/\varepsilon \in L^2_{\mathbb F}(\Omega;C(0,T;\mathbb R^n))$. By the linearity of \eqref{2.1T:Sys1} and \eqref{2.1T:Sys2}, $x^1(\cdot)$ is the unique solution to the following variational equation:
\begin{equation}\label{2.1T:VE}
\left\{
\begin{aligned}
& \mathrm dx^1 = \Big\{ Ax^1 +\bar A\mathbb E[x^1] +B_pu_p +\bar B_p \mathbb E[u_p] \Big\}\, \mathrm ds\\
& \qquad +\Big\{ Cx^1 +\bar C\mathbb E[x^1] +D_pu_p +\bar D_p \mathbb E[u_p] \Big\}\, \mathrm dW,\\
& x^1(0) =0.
\end{aligned}
\right.
\end{equation}

With the above introduced notations, we calculate
\[
J_p\big( u_p^*(\cdot) +\varepsilon u_p(\cdot), u_{-\{0,p\}}^*(\cdot); x_0, u_0(\cdot) \big) -J_p\big( u_p^*(\cdot), u_{-\{0,p\}}^*(\cdot); x_0, u_0(\cdot) \big) = \mbox{I} +\int_0^T \Big[ \mbox{II}(s) +\mbox{III}(s) \Big]\, \mathrm ds,
\]
where
\[
\begin{aligned}
\mbox{I} =\ & \mathbb E \Big[ \big\langle H_p\big( x^*(T) +\varepsilon x^1(T) \big),\ x^*(T) +\varepsilon x^1(T) \big\rangle -\big\langle H_p x^*(T),\ x^*(T) \big\rangle \Big]\\
& +\big\langle \bar H_p \mathbb E \big[x^*(T) +\varepsilon x^1(T)\big],\ \mathbb E \big[x^*(T) +\varepsilon x^1(T)\big] \big\rangle -\big\langle \bar H_p \mathbb E[x^*(T)],\ \mathbb E[x^*(T)] \big\rangle,\\
\mbox{II}(s) =\ & \mathbb E \Big[ \big\langle Q_p(s) \big( x^*(s) +\varepsilon x^1(s) \big),\ x^*(s) +\varepsilon x^1(s) \big\rangle -\big\langle Q_p(s) x^*(s),\ x^*(s) \big\rangle \Big]\\
& +\big\langle \bar Q_p(s) \mathbb E \big[x^*(s) +\varepsilon x^1(s)\big],\ \mathbb E \big[x^*(s) +\varepsilon x^1(s)\big] \big\rangle -\big\langle \bar Q_p(s) \mathbb E[x^*(s)],\ \mathbb E[x^*(s)] \big\rangle,\\
\mbox{III}(s) =\ & \mathbb E \Big[ \big\langle R_p(s) \big( u_p^*(s) +\varepsilon u_p(s) \big),\ u_p^*(s) +\varepsilon u_p(s) \big\rangle -\big\langle R_p(s) u_p^*(s),\ u_p^*(s) \big\rangle \Big]\\
& +\big\langle \bar R_p(s) \mathbb E \big[u_p^*(s) +\varepsilon u_p(s)\big],\ \mathbb E \big[u_p^*(s) +\varepsilon u_p(s)\big] \big\rangle -\big\langle \bar Q_p(s) \mathbb E[u_p^*(s)],\ \mathbb E[u_p^*(s)] \big\rangle.
\end{aligned}
\]
We continue to calculate
\[
\begin{aligned}
\mbox{I} =\ & 2\varepsilon \Big\{ \mathbb E \big[ \big\langle H_p x^*(T),\ x^1(T) \big\rangle \big] +\big\langle \bar H_p \mathbb E[x^*(T)],\ \mathbb E[x^1(T)] \big\rangle \Big\}\\
& +\varepsilon^2 \Big\{ \mathbb E\big[ \big\langle H_p x^1(T),\ x^1(T) \big\rangle \big] +\big\langle \bar H_p \mathbb E[x^1(T)],\ \mathbb E[x^1(T)] \big\rangle \Big\}\\
=\ & 2\varepsilon \mathbb E \Big[ \big\langle H_p x^*(T) +\bar H_p \mathbb E[x^*(T)],\ x^1(T) \big\rangle \Big]\\
& +\varepsilon^2 \Big\{ \mathbb E\Big[ \Big\langle H_p \big( x^1(T) -\mathbb E[x^1(T)]\big),\ x^1(T) -\mathbb E[x^1(T)] \Big\rangle \Big] +\Big\langle (H_p +\bar H_p) \mathbb E[x^1(T)],\ \mathbb E[x^1(T)] \Big\rangle \Big\}.
\end{aligned}
\]
Similar derivation can also be applied to $\mbox{II}(s)$ and $\mbox{III}(s)$. Therefore,
\[
J_p\big( u_p^*(\cdot) +\varepsilon u_p(\cdot), u_{-\{0,p\}}^*(\cdot); x_0, u_0(\cdot) \big) -J_p\big( u_p^*(\cdot), u_{-\{0,p\}}^*(\cdot); x_0, u_0(\cdot) \big) = 2\varepsilon \mbox{IV} +\varepsilon \mbox{V},
\]
where
\[
\begin{aligned}
\mbox{IV} =\ & \mathbb E\bigg\{ \big\langle H_p x^*(T) +\bar H_p \mathbb E[x^*(T)],\ x^1(T) \big\rangle +\int_0^T \Big[ \big\langle Q_px^* +\bar Q_p \mathbb E[x^*],\ x^1 \big\rangle\\
& +\big\langle R_pu^*_p +\bar R_p \mathbb E[u^*_p],\ u_p \big\rangle \Big]\, \mathrm ds \bigg\},\\
\mbox{V} =\ & \mathbb E\bigg\{ \Big\langle H_p\big( x^1(T) -\mathbb E[x^1(T)] \big),\ x^1(T) -\mathbb E[x^1(T)] \Big\rangle + \Big\langle (H_p +\bar H_p) \mathbb E[x^1(T)],\ \mathbb E[x^1(T)] \Big\rangle\\
& +\int_0^T \Big[ \Big\langle Q_p \big( x^1 -\mathbb E[x^1] \big),\ x^1 -\mathbb E[x^1] \Big\rangle +\Big\langle (Q_p +\bar Q_p) \mathbb E[x^1],\ \mathbb E[x^1] \Big\rangle \\
& +\Big\langle R_p \big( u_p -\mathbb E[u_p] \big),\ u_p -\mathbb E[u_p] \Big\rangle +\Big\langle (R_p +\bar R_p) \mathbb E[u_p],\ \mathbb E[u_p] \Big\rangle \Big]\, \mathrm ds \bigg\}.
\end{aligned}
\]
By Assumption (H1), $\mbox{V} \geq 0$. Therefore, due to the arbitrariness of $\varepsilon$, \eqref{2.1T:E1} is equivalent to the following statement:
\begin{equation}\label{2.1T:E2}
\mbox{IV} =0,\quad \mbox{for any } u_p(\cdot) \in L^2_{\mathbb F}(0,T;\mathbb R^{m_p}) \mbox{ and any } p=1,2,\dots,N.
\end{equation}

Now, we introduce the so-called adjoint equations to further simplify \eqref{2.1T:E2}. In fact, for any $p=1,2,\dots,N$, the introduced adjoint equation is an MF-BSDE as follows: 
\begin{equation}\label{2.1T:AE}
\left\{
\begin{aligned}
& \mathrm dy_p = - \Big\{ A^\top y_p +\bar A^\top \mathbb E[y_p] +C^\top z_p +\bar C^\top \mathbb E[z_p] +Q_px^* +\bar Q_p \mathbb E[x^*] \Big\}\, \mathrm ds +z_p\, \mathrm dW,\\
& y_p(T) = H_p x^*(T) +\bar H_p \mathbb E[x^*(T)].
\end{aligned}
\right.
\end{equation}
By Buckdahn et al. \cite{BDLP-09}, MF-BSDE \eqref{2.1T:AE} admits a unique solution $(y_p(\cdot), z_p(\cdot)) \in L^2_{\mathbb F}(\Omega;C(0,T;\mathbb R^n)) \times L^2_{\mathbb F}(0,T;\mathbb R^n)$. We apply It\^o's formula to $\langle x^1(\cdot),\ y_p(\cdot) \rangle$ on the interval $[0,T]$, where $x^1(\cdot)$ is the solution to MF-SDE \eqref{2.1T:VE}, to get
\[
\begin{aligned}
& \mathbb E\bigg\{ \big\langle H_p x^*(T) +\bar H_p \mathbb E[x^*(T)],\ x^1(T) \big\rangle +\int_0^T \big\langle Q_px^* +\bar Q_p \mathbb E[x^*],\ x^1 \big\rangle\, \mathrm ds \bigg\}\\
=\ & \mathbb E \int_0^T \Big[ \big\langle B_p^\top y_p +D^\top_p z_p,\ u_p \big\rangle +\big\langle \bar B_p^\top y_p +\bar D_p^\top z_p,\ \mathbb E[u_p] \big\rangle \Big]\, \mathrm ds \\
=\ & \mathbb E \int_0^T \Big\langle B_p^\top y_p +\bar B_p^\top \mathbb E[y_p] +D_p^\top z_p +\bar D_p^\top \mathbb E[z_p],\ u_p \Big\rangle\, \mathrm ds.
\end{aligned}
\]
By substituting the above equation into $\mbox{IV}$, we rewrite \eqref{2.1T:E2} as:
\[
\begin{aligned}
\mathbb E\int_0^T \Big\langle R_p u_p^* +\bar R_p \mathbb E[u_p^*] +B_p^\top y_p +\bar B_p^\top \mathbb E[y_p] +D_p^\top z_p +\bar D_p^\top \mathbb E[z_p],\ u_p \Big\rangle\, \mathrm ds =0,\\
\mbox{for any } u_p(\cdot) \in L^2_{\mathbb F}(0,T;\mathbb R^{m_p}) \mbox{ and any } p=1,2,\dots,N.
\end{aligned}
\]
Due to the arbitrariness of $u_p(\cdot)$, \eqref{2.1T:E2} continues to be equivalent to the following statement:
\begin{equation}\label{2.1T:E3}
R_p u_p^* +\bar R_p \mathbb E[u_p^*] +B_p^\top y_p +\bar B_p^\top \mathbb E[y_p] +D_p^\top z_p +\bar D_p^\top \mathbb E[z_p] =0,\quad \mbox{for any } p=1,2,\dots,N.
\end{equation}

Now, we use Notations \eqref{2:N3} and \eqref{2:N4} to rewrite \eqref{2.1T:AE} and \eqref{2.1T:E3} in concise forms:
\begin{equation}\label{2.1T:Adj}
\left\{
\begin{aligned}
& \mathrm dy_{-0} = -\Big\{ \widetilde A^\top y_{-0} +\bar {\widetilde A}^\top \mathbb E[y_{-0}] +\widetilde C^\top z_{-0} +\bar{\widetilde C}^\top \mathbb E[z_{-0}] +Qx^* +\bar Q \mathbb E[x^*] \Big\}\, \mathrm ds +z_{-0}\, \mathrm dW,\\
& y_{-0}(T) = Hx^*(T) +\bar H \mathbb E[x^*(T)],
\end{aligned}
\right.
\end{equation}
and
\begin{equation}\label{2.1T:Alg}
\begin{aligned}
& 0 = Ru_{-0}^* +\bar R\mathbb E[u_{-0}^*] +\widetilde B_{-0}^\top y_{-0} +\bar{\widetilde B}_{-0}^\top \mathbb E[y_{-0}] +\widetilde D_{-0}^\top z_{-0} +\bar{\widetilde D}_{-0}^\top \mathbb E[z_{-0}],
\end{aligned}
\end{equation}
respectively. We notice that \eqref{2.1T:Sys}, \eqref{2.1T:Adj} and \eqref{2.1T:Alg} just combine to form the Hamiltonian system \eqref{2:Ham}. From the previous analysis, we can draw the following conclusions:
\begin{enumerate}[(i).]
\item If $u_{-0}^*(\cdot)$ is a Nash equilibrium point of Problem (LQG) at $(x_0,u_{0}(\cdot))$, then by letting $x^*(\cdot)$ be the solution to MF-SDE \eqref{2.1T:Sys} and $(y_{-0}(\cdot), z_{-0}(\cdot))$ be the solution to MF-BSDE \eqref{2.1T:Adj} we get that $(u^*_{-0}(\cdot),(x^*(\cdot)^\top, y_{-0}(\cdot)^\top, z_{-0}(\cdot)^\top)^\top)$ is a solution to the Hamiltonian system \eqref{2:Ham}.
\item If $(u^*_{-0}(\cdot),(x^*(\cdot)^\top, y_{-0}(\cdot)^\top, z_{-0}(\cdot)^\top)^\top)$ is a solution to the Hamiltonian system \eqref{2:Ham}, then $u_{-0}^*(\cdot)$ is also a Nash equilibrium point of Problem (LQG) at $(x_0,u_{0}(\cdot))$.
\end{enumerate}

{\bf Step 2.} By the conclusions of Step 1, it is obvious that the existence of Nash equilibrium point of Problem (LQG) at $(x_0,u_0(\cdot))$ is equivalent to that of solution of Hamiltonian system \eqref{2:Ham}. The remaining thing is to prove that the uniqueness of the two is equivalent to each other.

Firstly, assume that the solution of the Hamiltonian system \eqref{2:Ham} is unique. Then, it is obvious that Conclusion (i) in Step 1 implies the uniqueness of Nash equilibrium point.

Secondly, assume that the Nash equilibrium point is unique. Let 
\[
(u^*_{-0}(\cdot), (x^*(\cdot)^\top, y_{-0}(\cdot)^\top, z_{-0}(\cdot)^\top)^\top) \quad \mbox{and} \quad (u^{*\prime}_{-0}(\cdot), (x^{*\prime}(\cdot)^\top, y_{-0}^{\prime}(\cdot)^\top, z_{-0}^\prime(\cdot)^\top)^\top)
\] 
be two solutions of the Hamiltonian system \eqref{2:Ham}. Then, Conclusion (ii) in Step 1 implies that $u^*_{-0}(\cdot) = u^{*\prime}_{-0}(\cdot)$. Based on this, the uniqueness of MF-SDE \eqref{2.1T:Sys} implies that $x^*(\cdot) = x^{*\prime}(\cdot)$. Furthermore, the uniqueness of MF-BSDE \eqref{2.1T:Adj} implies that $(y_{-0}(\cdot),z_{-0}(\cdot)) = (y^\prime_{-0}(\cdot), z^\prime_{-0}(\cdot))$. We have proved the uniqueness of solution of Hamiltonian system \eqref{2:Ham}.
\end{proof}

Now, we would like to introduce another assumption which is obviously weaker than Assumption (H1).

\medskip

\noindent{\bf Assumption (H2).} The matrix-valued functions $R_i(\cdot)^{-1}$ and $(R_i(\cdot) +\bar R_i(\cdot))^{-1}$ exist and are bounded for any $i=1,2,\dots,N$.

\medskip

In fact, under Assumption (H2), we can solve $u^*(\cdot)$ from the algebraic equation, i.e., the first equation, in the Hamiltonian system \eqref{2:Ham} as follows:
\begin{equation}\label{2:NEP}
\begin{aligned}
u^*_{-0} =\ & -R^{-1} \Big\{ \widetilde B_{-0}^\top \big( y_{-0} -\mathbb E[y_{-0}] \big) +\widetilde D_{-0}^\top \big(z_{-0} -\mathbb E[z_{-0}]\big) \Big\}\\
& -\big( R +\bar R \big)^{-1} \Big\{ \big( \widetilde B_{-0} +\bar{\widetilde B}_{-0} \big)^\top \mathbb E[y_{-0}] +\big( \widetilde D_{-0} +\bar{\widetilde D}_{-0} \big)^\top \mathbb E[z_{-0}] \Big\}.
\end{aligned}
\end{equation}
Then, substituting the above \eqref{2:NEP} into the Hamiltonian system \eqref{2:Ham} yields an MF-FBSDE:
\begin{equation}\label{2:MF-FBSDE}
\left\{
\begin{aligned}
& \mathrm dx^* = \Big\{ Ax^* +\bar A\mathbb E[x^*] + B^y_{-0} y_{-0} +\bar B^y_{-0} \mathbb E[y_{-0}] + B^z_{-0} z_{-0} +\bar B^z_{-0} \mathbb E[z_{-0}] +B_0 u_0 +\bar B_0 \mathbb E[u_0] \Big\}\, \mathrm ds\\
& \quad +\Big\{ Cx^* +\bar C\mathbb E[x^*] + D^y_{-0} y_{-0} +\bar D^y_{-0} \mathbb E[y_{-0}] + D^z_{-0} z_{-0} +\bar D^z_{-0} \mathbb E[z_{-0}] +D_0 u_0 +\bar D_0 \mathbb E[u_0] \Big\}\, \mathrm dW,\\
& \mathrm dy_{-0} = -\Big\{ Qx^* +\bar Q \mathbb E[x^*] +\widetilde A^\top y_{-0} +\bar {\widetilde A}^\top \mathbb E[y_{-0}] +\widetilde C^\top z_{-0} +\bar{\widetilde C}^\top \mathbb E[z_{-0}] \Big\}\, \mathrm ds +z_{-0}\, \mathrm dW,\\
& x^*(0) =x_0, \qquad y_{-0}(T) = Hx^*(T) +\bar H \mathbb E[x^*(T)],
\end{aligned}
\right.
\end{equation}
where we use the notations:
\begin{equation}\label{2:N6}
\begin{aligned}
& B^y_{-0} = -B_{-0}R^{-1}\widetilde B_{-0}^\top,  && 
\bar B^y_{-0} = B_{-0}R^{-1}\widetilde B_{-0}^\top -(B_{-0} +\bar B_{-0}) (R+\bar R)^{-1} (\widetilde B_{-0} +\bar{\widetilde B}_{-0})^\top,\\
& B^z_{-0} = -B_{-0}R^{-1}\widetilde D_{-0}^\top,  && 
\bar B^z_{-0} = B_{-0}R^{-1}\widetilde D_{-0}^\top -(B_{-0} +\bar B_{-0}) (R+\bar R)^{-1} (\widetilde D_{-0} +\bar{\widetilde D}_{-0})^\top,\\
& D^y_{-0} = -D_{-0}R^{-1}\widetilde B_{-0}^\top,  && 
\bar D^y_{-0} = D_{-0}R^{-1}\widetilde B_{-0}^\top -(D_{-0} +\bar D_{-0}) (R+\bar R)^{-1} (\widetilde B_{-0} +\bar{\widetilde B}_{-0})^\top,\\
& D^z_{-0} = -D_{-0}R^{-1}\widetilde D_{-0}^\top,  && 
\bar D^z_{-0} = D_{-0}R^{-1}\widetilde D_{-0}^\top -(D_{-0} +\bar D_{-0}) (R+\bar R)^{-1} (\widetilde D_{-0} +\bar{\widetilde D}_{-0})^\top.
\end{aligned}
\end{equation}
It is easy to verify the following result.

\begin{corollary}\label{2.1:Cor}
Let Assumption (H1) hold. Let $(x_0, u_0(\cdot)) \in \mathbb R^n \times L^2_{\mathbb F}(0,T;\mathbb R^{m_0})$ be given. Then, the existence and uniqueness of Nash equilibrium point of Problem (LQG) at $(x_0, u_0(\cdot))$ are equivalent to that of solution of MF-FBSDE \eqref{2:MF-FBSDE}, respectively. In this case, if $\theta(\cdot) \in \Theta$ is a solution to \eqref{2:MF-FBSDE}, then $u_{-0}^*(\cdot)$ defined by \eqref{2:NEP} provides a Nash equilibrium point of Problem (LQG) at $(x_0,u_0(\cdot))$.
\end{corollary}

\begin{remark}\label{2.1:Rem}
From the viewpoint of Corollary \ref{2.1:Cor}, when the independent variable $(x_0, u_0(\cdot))$ changes in the space $\mathbb R^n \times L^2_{\mathbb F}(0,T;\mathbb R^{m_0})$, the rational agents use the solution of MF-FBSDE \eqref{2:MF-FBSDE} parameterized by $(x_0, u_0(\cdot))$ to define their Nash equilibrium strategy point by point.
\end{remark}

\subsection{Exact controllability at the upper layer}

Now we consider the behavior of the regulator at the upper layer. She/he wants to steer the state $x(\cdot)$ from any existing initial value $x_0 \in \mathbb R^n$ to an arbitrary desired terminal random variable $x_T\in L^2_{\mathcal F_T}(\Omega;\mathbb R^n)$.

At this time, the regulator wisely realizes that the agents at the lower layer are rational, i.e., they will adopt a Nash equilibrium strategy to play games. By Corollary \ref{2.1:Cor} and Remark \ref{2.1:Rem}, the Nash equilibrium strategy is linked to the linear MF-FBSDE \eqref{2:MF-FBSDE}. For the convenience of later research, we would like to remove the initial condition from \eqref{2:MF-FBSDE} to get the following system: 
\begin{equation}\label{2:MF-GBCS}
\left\{
\begin{aligned}
& \mathrm dx^* = \Big\{ Ax^* +\bar A\mathbb E[x^*] + B^y_{-0} y_{-0} +\bar B^y_{-0} \mathbb E[y_{-0}] + B^z_{-0} z_{-0} +\bar B^z_{-0} \mathbb E[z_{-0}] +B_0 u_0 +\bar B_0 \mathbb E[u_0] \Big\}\, \mathrm ds\\
& \quad +\Big\{ Cx^* +\bar C\mathbb E[x^*] + D^y_{-0} y_{-0} +\bar D^y_{-0} \mathbb E[y_{-0}] + D^z_{-0} z_{-0} +\bar D^z_{-0} \mathbb E[z_{-0}] +D_0 u_0 +\bar D_0 \mathbb E[u_0] \Big\}\, \mathrm dW,\\
& \mathrm dy_{-0} = -\Big\{ Qx^* +\bar Q \mathbb E[x^*] +\widetilde A^\top y_{-0} +\bar {\widetilde A}^\top \mathbb E[y_{-0}] +\widetilde C^\top z_{-0} +\bar {\widetilde C}^\top \mathbb E[z_{-0}] \Big\}\, \mathrm ds +z_{-0}\, \mathrm dW,\\
& y_{-0}(T) = Hx^*(T) +\bar H \mathbb E[x^*(T)].
\end{aligned}
\right.
\end{equation}
Similar to Zhang and Guo \cite{ZhG-19AC, ZhG-19SICON} and Liu and Yu \cite{LY-22+}, we call the above \eqref{2:MF-GBCS} an MF-GBCS. We note that the coefficients of \eqref{2:MF-GBCS} involve $R(\cdot)^{-1}$ and $(R(\cdot) +\bar R(\cdot))^{-1}$ (see the definitions \eqref{2:N6}), then we always discuss \eqref{2:MF-GBCS} under Assumption (H2). Moreover, MF-GBCS \eqref{2:MF-GBCS} involves the agents' Nash equilibrium strategy and it is the essentially controlled system faced by the regulator.

\begin{definition}\label{2.2:Def}
MF-GBCS \eqref{2:MF-GBCS} is called {\rm exactly controllable} on $[0,T]$, if for any $(x_0, x_T) \in \mathbb R^n \times L^2_{\mathcal F_T}(\Omega;\mathbb R^n)$, there exists a pair of processes $(u_0(\cdot), \theta(\cdot)) \in L^2_{\mathbb F}(0,T;\mathbb R^{m_0}) \times \Theta$ (see the notations \eqref{2:N5}) satisfying MF-SGBCS \eqref{2:MF-GBCS} and the following initial and terminal conditions:
\begin{equation}\label{2:ITC}
x^*(0) =x_0,\quad x^*(T) = x_T.
\end{equation}
\end{definition}

\begin{remark}\label{2.2:Rem}
(i). The exact controllability of MF-GBCS \eqref{2:MF-GBCS} implies the existence of MF-FBSDE \eqref{2:MF-FBSDE} in the following sense. For any $x_0 \in \mathbb R^n$ and any supplementary $x_T\in L^2_{\mathcal F_T}(\Omega;\mathbb R^n)$, due to the exact controllability of \eqref{2:MF-GBCS}, there exists a pair $(u_0(\cdot), \theta(\cdot)) \in L^2_{\mathbb F}(0,T;\mathbb R^{m_{0}}) \times \Theta$ (which depends on $(x_0, x_T)$) to make MF-FBSDE \eqref{2:MF-FBSDE} satisfied (additionally, $x^*(T) =x_T$ is also satisfied). Therefore, for the above $x_0$ and $u_0(\cdot)$, MF-FBSDE \eqref{2:MF-FBSDE} admits a solution $\theta(\cdot)$. By Corollary \ref{2.1:Cor} and Remark \ref{2.1:Rem}, the exact controllability of MF-GBCS \eqref{2:MF-GBCS} further implies the existence of Nash equilibrium strategy of the agents.

(ii). We introduce

\medskip

\noindent{\bf Assumption (H3).} {\rm For any $(x_0, u_0(\cdot)) \in \mathbb R^n \times L^2_{\mathbb F}(0,T;\mathbb R^{m_0})$, MF-FBSDE \eqref{2:MF-FBSDE} has at most one solution in the space $\Theta$.}

\medskip

\noindent It is easy to verify that, if Assumption (H3) is true, then in Definition \ref{2.2:Def}, when $x_0$, $x_T$ and $u_0(\cdot)$ is given, the selection of $\theta(\cdot)$ is unique. Additionally, we notice that Assumption (H3) can be implied by Assumption (H1) in the special case of $N=1$ (see Tian and Yu \cite[Theorem 3.5]{TY-23} for example). But, how to remove Assumption (H3) when $N\geq 2$ is challenge and remains open at this stage.
\end{remark}

At the end of this subsection, we summarize the problems that the regulator needs to consider: (i) Judge whether MF-GBCS \eqref{2:MF-GBCS} is exactly controllable; (ii) If MF-GBCS \eqref{2:MF-GBCS} is exactly controllable, then for any existing initial state $x_0 \in \mathbb R^n$ and any desired terminal state $x_T \in L^2_{\mathcal F_T}(\Omega;\mathbb R^n)$, the regulator further looks for an admissible control $u_0(\cdot) \in L^2_{\mathbb F}(0,T;\mathbb R^{m_0})$ steering the corresponding state $x^*(\cdot)$ from $x_0$ to $x_T$. In the following Section \ref{SEC:Exact} and Section \ref{SEC:Steer}, we will try to solve these two problems respectively.

\subsection{An example of the MF-GBCS}

We try to give an example where the MF-GBCS framework may be applied.

\begin{example}\label{2.3:Example}
The principal-agent model is one of the common models in economic and financial activities (see \cite{S-08, W-15} for example). In this example, We begin with the Black-Scholes financial market model in which two assets are traded continuously. One is a risk-free bond with the interest rate $r>0$ and the other is a stock with the appreciation rate $\mu>r$ and the volatility $\sigma>0$. A principal (the regulator) employs two portfolio managers (the agents) to invest in the Black-Scholes financial market. Similar to Merton \cite{M-71} and El Karoui et al. \cite{EPQ-97}, the wealth process satisfies the following linear SDE:
\begin{equation}
\left\{
\begin{aligned}
& \mathrm dx(s) = \Big\{ rx(s) +(\mu-r) \big( \pi_1(s) +\pi_2(s) \big) +u_0(s) \Big\}\, \mathrm ds\\
& \hskip 14mm +\sigma \big( \pi_1(s) +\pi_2(s) \big)\, \mathrm dW(s),\quad s\in [0,T],\\
& x(0) =x_0,
\end{aligned}
\right.
\end{equation}
where $x_0 \in \mathbb R$ is the initial endowment, $\pi_i(s)$ is the dollar amount invested in the stock by Manager $i$ at time $s\in [0,T]$ ($i=1,2$), and $u_0(s)$ is the rate of money injected or withdrawn by the principal. The process $\pi_i(\cdot) \in L^2_{\mathbb F}(0,T;\mathbb R)$ (called the portfolio process in financial terminology) is actually the admissible control process of Manager $i$ ($i=1,2$). Similarly, $u_0(\cdot) \in L^2_{\mathbb F}(0,T;\mathbb R)$ is the admissible control process of the principal.

Denote $\pi(\cdot) = (\pi_1(\cdot), \pi_2(\cdot))^\top$. For any $(x_0, u_0(\cdot)) \in \mathbb R \times L^2_{\mathbb F}(0,T;\mathbb R)$, suppose that Manager $i$ aims to minimize the following cost functional:
\begin{equation}\label{2.3E:Cost}
\begin{aligned}
J_i\big( \pi(\cdot); x_0, u_0(\cdot) \big) = \mathbb E \bigg\{ & \alpha_i(x(T)) +\bar\alpha_i\big|x(T) -\mathbb E[x(T)]\big|^2 +\int_0^T \Big[ \beta_i(s,x(s))\\
& +\bar\beta_i \big| x(s) -\mathbb E[x(s)] \big|^2 +\gamma_i(s,\pi_i(s)) \Big]\, \mathrm ds \bigg\},\quad i=1,2.
\end{aligned}
\end{equation}
Here, $\gamma_i(s,\pi_i(s))$ represents the rate of the fees charged by securities regulatory commission. The cost functional \eqref{2.3E:Cost} is inspired by Liu and Yu \cite{LY-22+}, but two items $\bar\alpha_i \mbox{Var}\ [x(T)] = \bar\alpha_i \mathbb E[|x(T) -\mathbb E[x(T)]|^2]$ and $\bar \beta_i \int_0^T  \mbox{Var}\ [x(s)]\, \mathrm ds = \bar\beta_i \int_0^T \mathbb E[|x(s) -\mathbb E[x(s)]|^2]\, \mathrm ds$ are added to characterize the risk aversion (or risk seeking) of Manager $i$. According to the cost functionals \eqref{2.3E:Cost}, the two managers play an LQ Nash game.

Due to some financial planning, the principal at the upper layer hope to achieve a desired goal of terminal wealth by injecting or withdrawing money, i.e. the principal is faced with a problem of exact controllability.
\end{example}

\section{Exact controllability of MF-GBCS}\label{SEC:Exact}

When we consider the exact controllability of MF-GBCS \eqref{2:MF-GBCS}, unlike deterministic systems, the arbitrary terminal state $x_T$ belongs to an infinite dimensional space $L^2_{\mathcal F_T}(\Omega;\mathbb R^n)$, which makes the problem difficult. In order to overcome the difficulty caused by the infinite dimension, similar to Peng \cite{P-94} (see also \cite{LP-10, WYYY-17, Y-21, LY-22+}), we will adopt a ``backward'' viewpoint and method.

\subsection{Backward system and exact null-controllability}

First of all, similar to Peng \cite[Definition 3.1]{P-94} and Yu \cite[Definition 3.2]{Y-21} respectively, we introduce a couple of concepts as follows.

\begin{definition}
(i). MF-GBCS \eqref{2:MF-GBCS} is called {\rm exactly terminal-controllable} on $[0,T]$, if for any $x_T \in L^2_{\mathcal F_T}(\Omega;\mathbb R^n)$, there exists a pair of processes $(u_0(\cdot), \theta(\cdot)) \in L^2_{\mathbb F}(0,T;\mathbb R^{m_0}) \times \Theta$ satisfying MF-GBCS \eqref{2:MF-GBCS} and the following terminal condition:
\begin{equation}\label{3:TC}
x^*(T) =x_T.
\end{equation}
(ii). MF-GBCS \eqref{2:MF-GBCS} is called {\rm exactly null-controllable} on $[0,T]$, if for any $x_0 \in \mathbb R^n$, there exists a pair of processes $(u_0(\cdot), \theta(\cdot)) \in L^2_{\mathbb F}(0,T;\mathbb R^{m_0}) \times \Theta$ satisfying MF-GBCS \eqref{2:MF-GBCS} and the following initial-terminal condition:
\begin{equation}\label{3:IT0C}
x^*(0) =x_0,\quad x^*(T) =0.
\end{equation}
\end{definition}

Obviously, both the exact terminal-controllability and the exact null-controllability are weaker than the exact controllability (see Definition \ref{2.2:Def}). However, the following proposition shows that these two weaker concepts together are equivalent to the exact controllability.

\begin{proposition}\label{3.1P:E}
Under Assumption (H2), MF-GBCS \eqref{2:MF-GBCS} is exactly controllable if and only if, it is both exactly terminal-controllable and exactly null-controllable.
\end{proposition}

\begin{proof}
The necessity is obvious. Then, we only need prove the sufficiency. For any $(x_0, x_T) \in \mathbb R^n \times L^2_{\mathcal F_T}(\Omega;\mathbb R^n)$, since MF-GBCS \eqref{2:MF-GBCS} is exactly terminal-controllable, then there exists $(u_0^1(\cdot),\theta^1(\cdot)) \in L^2_{\mathbb F}(0,T;\mathbb R^{m_0}) \times \Theta$ such that
{\small
\[
\left\{
\begin{aligned}
& \mathrm dx^{*1} = \Big\{ Ax^{*1} +\bar A\mathbb E[x^{*1}] + B^y_{-0} y_{-0}^1 +\bar B^y_{-0} \mathbb E[y_{-0}^1] + B^z_{-0} z_{-0}^1 +\bar B^z_{-0} \mathbb E[z_{-0}^1] +B_0 u_0^1 +\bar B_0 \mathbb E[u_0^1] \Big\}\, \mathrm ds\\
& \quad +\Big\{ Cx^{*1} +\bar C\mathbb E[x^{*1}] + D^y_{-0} y_{-0}^1 +\bar D^y_{-0} \mathbb E[y_{-0}^1] + D^z_{-0} z_{-0}^1 +\bar D^z_{-0} \mathbb E[z_{-0}^1] +D_0 u_0^1 +\bar D_0 \mathbb E[u_0^1] \Big\}\, \mathrm dW,\\
& \mathrm dy_{-0}^1 = -\Big\{ Qx^{*1} +\bar Q \mathbb E[x^{*1}] +\widetilde A^\top y_{-0}^1 +\bar {\widetilde A}^\top \mathbb E[y_{-0}^1] +\widetilde C^\top z_{-0}^1 +\bar {\widetilde C}^\top \mathbb E[z_{-0}^1] \Big\}\, \mathrm ds +z_{-0}^1\, \mathrm dW,\\
& y_{-0}^1(T) = Hx^{*1}(T) +\bar H \mathbb E[x^{*1}(T)],\qquad x^{*1}(T) =x_T.
\end{aligned}
\right.
\]}
On the other hand, since MF-GBCS \eqref{2:MF-GBCS} is also exactly null-controllable, then there exists $(u_0^2(\cdot), \theta^2(\cdot)) \in L^2_{\mathbb F}(0,T;\mathbb R^{m_0}) \times \Theta$ such that
{\small
\[
\left\{
\begin{aligned}
& \mathrm dx^{*2} = \Big\{ Ax^{*2} +\bar A\mathbb E[x^{*2}] + B^y_{-0} y_{-0}^2 +\bar B^y_{-0} \mathbb E[y_{-0}^2] + B^z_{-0} z_{-0}^2 +\bar B^z_{-0} \mathbb E[z_{-0}^2] +B_0 u_0^2 +\bar B_0 \mathbb E[u_0^2] \Big\}\, \mathrm ds\\
& \quad +\Big\{ Cx^{*2} +\bar C\mathbb E[x^{*2}] + D^y_{-0} y_{-0}^2 +\bar D^y_{-0} \mathbb E[y_{-0}^2] + D^z_{-0} z_{-0}^2 +\bar D^z_{-0} \mathbb E[z_{-0}^2] +D_0 u_0^2 +\bar D_0 \mathbb E[u_0^2] \Big\}\, \mathrm dW,\\
& \mathrm dy_{-0}^2 = -\Big\{ Qx^{*2} +\bar Q \mathbb E[x^{*2}] +\widetilde A^\top y_{-0}^2 +\bar {\widetilde A}^\top \mathbb E[y_{-0}^2] +\widetilde C^\top z_{-0}^2 +\bar {\widetilde C}^\top \mathbb E[z_{-0}^2] \Big\}\, \mathrm ds +z_{-0}^2\, \mathrm dW,\\
& y_{-0}^2(T) = Hx^{*2}(T) +\bar H \mathbb E[x^{*2}(T)],\qquad x^{*2}(0) =x_0 -x^{*1}(0),\qquad x^{*2}(T) =0.
\end{aligned}
\right.
\]}
Due to the linearity of MF-GBCS \eqref{2:MF-GBCS}, the pair of processes
\[
\big(u_0(\cdot),\ \theta(\cdot)\big) := \Big( u_0^1(\cdot) +u_0^2(\cdot),\ \theta^1(\cdot) +\theta^2(\cdot) \Big) \in L^2_{\mathbb F}(0,T;\mathbb R^{m_0}) \times \Theta
\]
satisfies
\[
\left\{
\begin{aligned}
& \mathrm dx^* = \Big\{ Ax^* +\bar A\mathbb E[x^*] + B^y_{-0} y_{-0} +\bar B^y_{-0} \mathbb E[y_{-0}] + B^z_{-0} z_{-0} +\bar B^z_{-0} \mathbb E[z_{-0}] +B_0 u_0 +\bar B_0 \mathbb E[u_0] \Big\}\, \mathrm ds\\
& \quad +\Big\{ Cx^* +\bar C\mathbb E[x^*] + D^y_{-0} y_{-0} +\bar D^y_{-0} \mathbb E[y_{-0}] + D^z_{-0} z_{-0} +\bar D^z_{-0} \mathbb E[z_{-0}] +D_0 u_0 +\bar D_0 \mathbb E[u_0] \Big\}\, \mathrm dW,\\
& \mathrm dy_{-0} = -\Big\{ Qx^* +\bar Q \mathbb E[x^*] +\widetilde A^\top y_{-0} +\bar {\widetilde A}^\top \mathbb E[y_{-0}] +\widetilde C^\top z_{-0} +\bar {\widetilde C}^\top \mathbb E[z_{-0}] \Big\}\, \mathrm ds +z_{-0}\, \mathrm dW,\\
& y_{-0}(T) = Hx^*(T) +\bar H \mathbb E[x^*(T)],\qquad x^*(0) = x_0, \qquad x^*(T) =x_T.
\end{aligned}
\right.
\]
We finish the proof.
\end{proof}

From the viewpoint of Proposition \ref{3.1P:E}, the study of exact controllability can be broken down into the study of exact terminal-controllability and exact null-controllability. We note that the exact null-controllability only requires the terminal state to reach $0$, so there is no the infinite dimension difficulty of ``any terminal state belongs to $L^2_{\mathcal F_T}(\Omega;\mathbb R^n)$''. However, MF-GBCS \eqref{2:MF-GBCS} is not always exactly terminal-controllable. Next, we will give a sufficient condition, under which we can rewrite MF-GBCS \eqref{2:MF-GBCS} as a ``backward system'', and apply the conclusion of MF-BSDEs to obtain the exact terminal-controllability.

We introduce the following

\medskip

\noindent{\bf Assumption (H4).} $D_0(\cdot) D_0(\cdot)^\top \gg 0$ and $(D_0(\cdot) +\bar D_0(\cdot))(D_0(\cdot) +\bar D_0(\cdot))^\top \gg 0$.

\medskip

\noindent Obviously, Assumption (H4) implies that the dimension of the regulator's control is bigger than or equal to that of the state, i.e. $m_0\geq n$. Under Assumption (H4), $[D_0(\cdot) D_0(\cdot)^\top]^{-1}$ and $[(D_0(\cdot) +\bar D_0(\cdot))(D_0(\cdot) +\bar D_0(\cdot))^\top]^{-1}$ exist and are bounded. Especially, when $D_0(\cdot)$ and $D_0(\cdot) +\bar D_0(\cdot)$ are time-invariant, Assumption (H4) is equivalent to

\medskip

\noindent{\bf Assumption (H4$'$).} $\mbox{Rank } D_0 =n$ and $\mbox{Rank }(D_0 +\bar D_0) =n$.

\medskip

\noindent Here the notation $\mbox{Rank } \mathbb A$ denotes the rank of matrix $\mathbb A$.  We notice that, when Goreac \cite{Gor-14} and Yu \cite{Y-21} studied the exact terminal-controllability and the exact controllability of MF-SDE systems respectively, sufficient conditions similar to Assumption (H4$'$) or Assumption (H4) was also introduced. In the present paper, we shall work under Assumption (H4) to investigate the exact controllability of MF-GBCS \eqref{2:MF-GBCS}.

For convenience, we use the following notation convention in this paper: Let $M$, $\bar M$ and $\widehat M$ be three matrices or matrix-valued functions with the same dimension. Then they always satisfy the relationship
\begin{equation}\label{3:N1}
\widehat M = M +\bar M.
\end{equation}
Moreover, under Assumption (H4), we define (the argument $s$ is suppressed for simplicity)
\begin{equation}\label{3:N2}
\begin{aligned}
& \mathbf A_1 = A -B_0D_0^\top (D_0D_0^\top)^{-1}C, \quad
&& \widehat {\mathbf A}_1 = \widehat A - \widehat B_0 \widehat D_0^\top (\widehat D_0 \widehat D_0^\top)^{-1} \widehat C,\\
& \mathbf A_2 = B^y_{-0} -B_0D_0^\top(D_0D_0^\top)^{-1}D^y_{-0}, \quad
&& \widehat{\mathbf A}_2 = \widehat B^y_{-0} -\widehat B_0 \widehat D_0^\top( \widehat D_0 \widehat D_0^\top)^{-1} \widehat D^y_{-0},\\
& \mathbf C_1 = B_0 D_0^\top (D_0D_0^\top)^{-1}, \quad
&& \widehat{\mathbf C}_1 = \widehat B_0 \widehat D_0^\top (\widehat D_0 \widehat D_0^\top)^{-1},\\
& \mathbf C_2 =B^z_{-0} -B_0D_0^\top (D_0D_0^\top)^{-1}D^z_{-0}, \quad
&& \widehat{\mathbf C}_2 =\widehat B^z_{-0} -\widehat B_0 \widehat D_0^\top (\widehat D_0 \widehat D_0^\top)^{-1} \widehat D^z_{-0},\\
& \mathbf B_1 = B_0 \big[ I -D_0^\top(D_0D_0^\top)^{-1}D_0 \big], \quad
&& \widehat{\mathbf B}_1 = \widehat B_0 \big[ I - \widehat D_0^\top( \widehat D_0 \widehat D_0^\top)^{-1} \widehat D_0 \big],
\end{aligned}
\end{equation}
where $B^y_{-0}$, $D^y_{-0}$, $B^z_{-0}$, $D^z_{-0}$, $\widehat B^y_{-0}$, $\widehat D^y_{-0}$, $\widehat B^z_{-0}$, $\widehat D^z_{-0}$ are given by \eqref{2:N6} and the above \eqref{3:N1}. We continue to introduce
\begin{equation}\label{3:N3}
\begin{aligned}
& \mathbf A = \begin{pmatrix} \mathbf A_1 & \mathbf A_2 \\ -Q & -\widetilde A^\top \end{pmatrix}, \quad 
&& \mathbf C = \begin{pmatrix} \mathbf C_1 & \mathbf C_2 \\ 0 & -\widetilde C^\top \end{pmatrix}, \quad
&& \mathbf B = \begin{pmatrix} \mathbf B_1 \\ 0 \end{pmatrix},\\
& \bar{\mathbf A} = \begin{pmatrix} \bar{\mathbf A}_1 & \bar{\mathbf A}_2 \\ -\bar Q & -\bar{\widetilde A}^\top \end{pmatrix}, \quad 
&& \bar{\mathbf C} = \begin{pmatrix} \bar{\mathbf C}_1 & \bar{\mathbf C}_2 \\ 0 & -\bar{\widetilde C}^\top \end{pmatrix}, \quad
&& \bar{\mathbf B} = \begin{pmatrix} \bar{\mathbf B}_1 \\ 0 \end{pmatrix},
\end{aligned}
\end{equation}
where $Q$, $\bar Q$, $\widetilde A$, $\bar{\widetilde A}$, $\widetilde C$, $\bar{\widetilde C}$ are given by \eqref{2:N3}. With these notations, we introduce the following system:
\begin{equation}\label{3:BS}
\left\{
\begin{aligned}
& \mathrm dy = \Big\{ \mathbf A y +\bar{\mathbf A} \mathbb E[y] + \mathbf C z +\bar{\mathbf C} \mathbb E[z] +\mathbf B v +\bar{\mathbf B} \mathbb E[v] \Big\}\, \mathrm ds +z\, \mathrm dW, \quad s\in [0,T],\\
& y_{-0}(T) = Hx^*(T) +\bar H \mathbb E[x^*(T)],
\end{aligned}
\right.
\end{equation}
where we has denoted
\begin{equation}\label{3:N4}
y(\cdot) = \begin{pmatrix} x^*(\cdot)^\top & y_{-0}(\cdot)^\top \end{pmatrix}^\top \quad \mbox{and} \quad 
z(\cdot) = \begin{pmatrix} q(\cdot)^\top & z_{-0}(\cdot)^\top \end{pmatrix}^\top.
\end{equation}
Similar to \eqref{2:N5}, here we also denote
\begin{equation}\label{3:N5}
\begin{aligned}
& \pi(\cdot) = \begin{pmatrix} y(\cdot)^\top & z(\cdot)^\top \end{pmatrix}^\top,\\
& \Pi = L^2_{\mathbb F}(\Omega;C(0,T;\mathbb R^{(1+N)n})) \times L^2_{\mathbb F}(0,T;\mathbb R^{(1+N)n}).
\end{aligned}
\end{equation}

\begin{theorem}\label{3T-Equiv}
Let Assumptions (H2) and (H4) hold. Then MF-GBCS \eqref{2:MF-GBCS} is equivalent to System \eqref{3:BS} in the following sense:

\smallskip

\noindent (i). If $(u_0(\cdot), \theta(\cdot)) \in L^2_{\mathbb F}(0,T;\mathbb R^{m_0}) \times \Theta$ satisfies MF-GBCS \eqref{2:MF-GBCS}, then by letting
\begin{equation}\label{3T:u02qv}
\left\{
\begin{aligned}
q =\ & C\big( x^* -\mathbb E[x^*] \big) +D^y_{-0}\big( y_{-0} -\mathbb E[y_{-0}] \big) +D^z_{-0}\big( z_{-0} -\mathbb E[z_{-0}] \big) +D_0 \big( u_0 -\mathbb E[u_0]\big)\\
& +\widehat C \mathbb E[x^*] +\widehat D^y_{-0} \mathbb E[y_{-0}] +\widehat D^z_{-0} \mathbb E[z_{-0}] +\widehat D_0 \mathbb E[u_0],\\
v =\ & u_0,
\end{aligned}
\right.
\end{equation}
the pair of processes $(v(\cdot),\pi(\cdot)) \in L^2_{\mathbb F}(0,T;\mathbb R^{m_0}) \times \Pi$ satisfies System \eqref{3:BS};

\smallskip

\noindent (ii). If $(v(\cdot),\pi(\cdot)) \in L^2_{\mathbb F}(0,T;\mathbb R^{m_0}) \times \Pi$ satisfies System \eqref{3:BS}, then by letting
\begin{equation}\label{3T:qv2u0}
\begin{aligned}
u_0 =\ & D_0^\top (D_0D_0^\top)^{-1} \Big\{ \big( q-\mathbb E[q] \big) -C\big( x^* -\mathbb E[x^*] \big) -D^y_{-0} \big( y_{-0} -\mathbb E[y_{-0}] \big)\\
& -D^z_{-0} \big( z_{-0} -\mathbb E[z_{-0}] \big) \Big\} +\big[ I -D_0^\top (D_0D_0^\top)^{-1}D_0 \big] \big( v -\mathbb E[v] \big) +\widehat D_0^\top (\widehat D_0 \widehat D_0^\top)^{-1}\\
& \times \Big\{ \mathbb E[q] -\widehat C \mathbb E[x^*] -\widehat D^y_{-0} \mathbb E[y_{-0}] -\widehat D^z_{-0} \mathbb E[z_{-0}] \Big\} +\big[ I -\widehat D_0^\top (\widehat D_0 \widehat D_0^\top)^{-1}\widehat D_0 \big] \mathbb E[v],
\end{aligned}
\end{equation}
the pair of processes $(u_0(\cdot), \theta(\cdot)) \in L^2_{\mathbb F}(0,T;\mathbb R^{m_0}) \times \Theta$ satisfies MF-GBCS \eqref{2:MF-GBCS}.
\end{theorem}

\begin{proof}
(i) Since $(u_0(\cdot), \theta(\cdot)) \in L^2_{\mathbb F}(0,T;\mathbb R^{m_0}) \times \Theta$, according to \eqref{3T:u02qv}, it is easy to verify that $(q(\cdot),v(\cdot)) \in L^2_{\mathbb F}(0,T;\mathbb R^n) \times L^2_{\mathbb F}(0,T;\mathbb R^{m_0})$, then $\pi(\cdot) \in \Pi$. Next, we shall prove that $(v(\cdot),\pi(\cdot))$ satisfies System \eqref{3:BS}.

Clearly, \eqref{3T:u02qv} is equivalent to 
\[
\left\{
\begin{aligned}
& \mathbb E[q] = \widehat C \mathbb E[x^*] +\widehat D^y_{-0} \mathbb E[y_{-0}] +\widehat D^z_{-0} \mathbb E[z_{-0}] +\widehat D_0 \mathbb E[u_0],\\
& q -\mathbb E[q] = C\big( x^* -\mathbb E[x^*] \big) +D^y_{-0}\big( y_{-0} -\mathbb E[y_{-0}] \big) +D^z_{-0}\big( z_{-0} -\mathbb E[z_{-0}] \big) +D_0 \big( u_0 -\mathbb E[u_0]\big),\\
& \mathbb E[v] = \mathbb E[u_0], \qquad v -\mathbb E[v] = u_0 -\mathbb E[u_0].
\end{aligned}
\right.
\]
Based on this, on the one hand, we calculate
\begin{equation}\label{3T:u02}
\begin{aligned}
& \mathbb E[u_0] = \widehat D_0^\top (\widehat D_0\widehat D_0^\top)^{-1} \widehat D_0 \mathbb E[u_0] +\big[ I -\widehat D_0^\top (\widehat D_0\widehat D_0^\top)^{-1} \widehat D_0 \big] \mathbb E[u_0]\\
=\ & \widehat D_0^\top (\widehat D_0\widehat D_0^\top)^{-1} \Big\{ \mathbb E[q] - \widehat C \mathbb E[x^*] -\widehat D^y_{-0} \mathbb E[y_{-0}] -\widehat D^z_{-0} \mathbb E[z_{-0}] \Big\} +\big[ I -\widehat D_0^\top (\widehat D_0\widehat D_0^\top)^{-1} \widehat D_0 \big] \mathbb E[v].
\end{aligned}
\end{equation}
With the help of Notation \eqref{3:N2}, we derive
\[
\begin{aligned}
\widehat B_0 \mathbb E[u_0] =\ & -\widehat B_0 \widehat D_0^\top (\widehat D_0 \widehat D_0^\top)^{-1} \widehat C \mathbb E[x^*] -\widehat B_0\widehat D_0^\top(\widehat D_0\widehat D_0^\top)^{-1} \widehat D^y_{-0} \mathbb E[y_{-0}] +\widehat B_0 \widehat D_0^\top (\widehat D_0\widehat D_0^\top)^{-1} \mathbb E[q]\\
& -\widehat B_0\widehat D_0^\top (\widehat D_0 \widehat D_0^\top)^{-1} \widehat D^z_{-0} \mathbb E[z_{-0}] +\widehat B_0 \big[ I -\widehat D_0^\top(\widehat D_0 \widehat D_0^\top)^{-1}\widehat D_0 \big]\mathbb E[v]\\
=\ & \big( \widehat{\mathbf A}_1 -\widehat A \big) \mathbb E[x^*] +\big( \widehat{\mathbf A}_2 -\widehat B^y_{-0} \big) \mathbb E[y_{-0}] +\widehat{\mathbf C}_1 \mathbb E[q] + \big( \widehat{\mathbf C}_2 -\widehat B^z_{-0} \big) \mathbb E[z_{-0}] +\widehat{\mathbf B}_1 \mathbb E[v].
\end{aligned}
\]
Therefore,
\begin{equation}\label{3T:FDr2}
\begin{aligned}
& \widehat A \mathbb E[x^*] +\widehat B^y_{-0} \mathbb E[y_{-0}] +\widehat B^z_{-0} \mathbb E[z_{-0}] +\widehat B_0 \mathbb E[u_0]\\
=\ & \widehat{\mathbf A}_1 \mathbb E[x^*] +\widehat{\mathbf A}_2 \mathbb E[y_{-0}] +\widehat{\mathbf C}_1 \mathbb E[q] +\widehat{\mathbf C}_2 \mathbb E[z_{-0}] +\widehat{\mathbf B}_1 \mathbb E[v].
\end{aligned}
\end{equation}
On the other hand, we can similarly have
\begin{equation}\label{3T:u01}
\begin{aligned}
u_0 -\mathbb E[u_0] =\ & D_0^\top (D_0D_0^\top)^{-1} \Big\{ \big( q -\mathbb E[q] \big) -C \big( x^* -\mathbb E[x^*] \big) -D^y_{-0} \big( y_{-0} -\mathbb E[y_{-0}] \big) \\
& -D^z_{-0} \big( z_{-0} -\mathbb E[z_{-0}] \big) \Big\} +\big[ I -D_0^\top (D_0D_0^\top)^{-1} D_0 \big] \big( v-\mathbb E[v] \big)
\end{aligned}
\end{equation}
and
\begin{equation}\label{3T:FDr1}
\begin{aligned}
& A \big( x^* -\mathbb E[x^*] \big) +B^y_{-0} \big( y_{-0} -\mathbb E[y_{-0}] \big) +B^z_{-0} \big( z_{-0} -\mathbb E[z_{-0}] \big) +B_0 \big( u_0 -\mathbb E[u_0] \big)\\
=\ & \mathbf A_1 \big( x^* -\mathbb E[x^*] \big) +\mathbf A_2 \big( y_{-0} -\mathbb E[y_{-0}] \big) +\mathbf C_1 \big( q-\mathbb E[q] \big) +\mathbf C_2 \big( z_{-0} -\mathbb E[z_{-0}] \big) +\mathbf B_1 \big(v -\mathbb E[v]\big).
\end{aligned}
\end{equation}

By \eqref{3T:u02qv}, \eqref{3T:FDr2} and \eqref{3T:FDr1} (also with the help of the notation convention \eqref{3:N1}), we rewrite MF-GBCS \eqref{2:MF-GBCS} as
\begin{equation}\label{3T:BS}
\left\{
\begin{aligned}
& \mathrm dx^* = \Big\{ \mathbf A_1 x^* +\bar{\mathbf A}_1 \mathbb E[x^*] +\mathbf A_2 y_{-0} +\bar{\mathbf A}_2 \mathbb E[y_{-0}] +\mathbf C_1 q +\bar{\mathbf C}_1 \mathbb E[q] +\mathbf C_2 z_{-0} +\bar{\mathbf C}_2 \mathbb E[z_{-0}]\\
& \qquad + \mathbf B_1 v +\bar{\mathbf B}_1 \mathbb E[v] \Big\}\, \mathrm ds +q\, \mathrm dW,\\
& \mathrm dy_{-0} = \Big\{ -Qx^* -\bar Q \mathbb E[x^*] -\widetilde A^\top y_{-0} -\bar {\widetilde A}^\top \mathbb E[y_{-0}] -\widetilde C^\top z_{-0} -\bar {\widetilde C}^\top \mathbb E[z_{-0}] \Big\}\, \mathrm ds +z_{-0}\, \mathrm dW,\\
& y_{-0}(T) = Hx^*(T) +\bar H \mathbb E[x^*(T)].
\end{aligned}
\right.
\end{equation}
By noticing Notations \eqref{3:N3} and \eqref{3:N4}, the above system happens to be \eqref{3:BS}.

(ii) Due to $(v(\cdot),\pi(\cdot)) \in L^2_{\mathbb F}(0,T;\mathbb R^{m_0})\times \Pi$, it is clear that \eqref{3T:qv2u0} implies that $(u_0(\cdot),\theta(\cdot)) \in L^2_{\mathbb F}(0,T;\mathbb R^{m_0})\times \Theta$. The remaining thing is to prove that $(u_0(\cdot), \theta(\cdot))$ satisfies MF-GBCS \eqref{2:MF-GBCS}.

First of all, we know that \eqref{3T:qv2u0} is equivalent to \eqref{3T:u02} and \eqref{3T:u01}. Moreover, from the derivation in Step (i), Equations \eqref{3T:FDr2} and \eqref{3T:FDr1} still hold true. Furthermore, we calculate from \eqref{3T:u02} and \eqref{3T:u01} to get
\[
\widehat D_0 \mathbb E[u_0] = \mathbb E[q] -\widehat C \mathbb E[x^*] -\widehat D^y_{-0} \mathbb E[y_{-0}] -\widehat D^z_{-0} \mathbb E[z_{-0}]
\]
and
\[
D_0 \big( u_0 -\mathbb E[u_0] \big) = \big(q -\mathbb E[q]\big) -C\big( x^* -\mathbb E[x^*]\big) -D^y_{-0} \big( y_{-0} -\mathbb E[y_{-0}] \big) -D^z_{-0} \big(z_{-0} -\mathbb E[z_{-0}]\big),
\]
respectively. Therefore,
\begin{equation}\label{3T:FDi}
q = Cx^* +\bar C\mathbb E[x^*] + D^y_{-0} y_{-0} +\bar D^y_{-0} \mathbb E[y_{-0}] + D^z_{-0} z_{-0} +\bar D^z_{-0} \mathbb E[z_{-0}] +D_0 u_0 +\bar D_0 \mathbb E[u_0].
\end{equation}
By \eqref{3T:FDr2}, \eqref{3T:FDr1} and \eqref{3T:FDi}, we derive MF-GBCS \eqref{2:MF-GBCS} from \eqref{3T:BS} (which is equivalent to \eqref{3:BS}). The proof is completed.
\end{proof}

According to Theorem \ref{3T-Equiv}, we can equivalently transfer the research from MF-GBCS \eqref{2:MF-GBCS} to System \eqref{3:BS}.

\begin{definition}
System \eqref{3:BS} is called {\rm exactly controllable} (resp. {\rm exactly terminal-controllable}, {\rm exactly null-controllable}) on $[0,T]$, if for any $(x_0, x_T) \in \mathbb R^n \times L^2_{\mathcal F_T}(\Omega;\mathbb R^n)$ (resp. $x_T \in L^2_{\mathcal F_T}(\Omega;\mathbb R^n)$, $x_0 \in \mathbb R^n$), there exists a pair of processes $(v(\cdot), \pi(\cdot)) \in L^2_{\mathbb F}(0,T;\mathbb R^{m_0}) \times \Pi$ satisfying System \eqref{3:BS} and the initial-terminal condition \eqref{2:ITC} (resp. the terminal condition \eqref{3:TC}, the initial-terminal condition \eqref{3:IT0C}).
\end{definition}

\begin{corollary}\label{3.1C-GBeB}
Under Assumptions (H2) and (H4), the exact controllability (resp. the exact terminal-controllability, the exact null-controllability) of MF-GBCS \eqref{2:MF-GBCS} is equivalent to that of System \eqref{3:BS}.
\end{corollary}

\noindent Moreover, we easily verify that the proof of Proposition \ref{3.1P:E} only depends on the linearity of the system. Then it can be applied to System \eqref{3:BS} to yield

\begin{proposition}\label{3.1P-BCeNC}
Under Assumptions (H2) and (H4), System \eqref{3:BS} is exactly controllable if and only if, it is both exactly terminal-controllable and exactly null-controllable.
\end{proposition}

We notice that the system \eqref{3:BS} and the terminal condition \eqref{3:TC} together form an MF-BSDE. Due to this, we would like to call \eqref{3:BS} a {\it backward system}. By the result of MF-BSDEs (see Buckdahn et al. \cite[Theorem 3.1]{BDLP-09} or Tian and Yu \cite[Proposition 2.2]{TY-23}), for any $x_T \in L^2_{\mathcal F_T}(\Omega;\mathbb R^n)$ and any $v(\cdot) \in L^2_{\mathbb F}(0,T;\mathbb R^{m_0})$, MF-BSDE \eqref{3:BS} and \eqref{3:TC} admits a unique solution $\pi(\cdot) \in \Pi$, which implies that the backward system \eqref{3:BS} is exactly terminal-controllable. Based on this, Corollary \ref{3.1C-GBeB} and Proposition \ref{3.1P-BCeNC} imply the following

\begin{corollary}\label{3.1C-3E}
Let Assumptions (H2) and (H4) hold. Then the following three statements are equivalent:
\begin{enumerate}[(i).]
\item MF-GBCS \eqref{2:MF-GBCS} is exactly controllable;
\item The backward system \eqref{3:BS} is exactly controllable;
\item The backward system \eqref{3:BS} is exactly null-controllable.
\end{enumerate}
\end{corollary}

\begin{remark}\label{3.1:RMK}
Let Assumptions (H2) and (H4) hold.
\begin{enumerate}[(i).]
\item Corollary \ref{3.1C-3E} indicates that the terminal coefficients $H$ and $\bar H$ do not affect the exact controllability of the backward system \eqref{3:BS} (or MF-GBCS \eqref{2:MF-GBCS}).
\item When $m_0 =n$, i.e. the dimension of the regulator's control is equal to that of the state, the coefficients $D_0(\cdot)$ and $\widehat D_0(\cdot)$ are $(n\times n)$-dimensional and non-singular. By the definition \eqref{3:N2},
\[
\mathbf B_1 =B_0 \big[ I -D_0^\top (D_0D_0^\top)^{-1}D_0 \big]
= B_0 \big[ I -D_0^\top (D_0^\top)^{-1} D_0^{-1} D_0 \big] =0.
\]
Similarly, $\widehat{\mathbf B}_1(\cdot) =0$. Then \eqref{3:N3} further implies that $\mathbf B(\cdot) =\bar{\mathbf B}(\cdot) =0$, i.e. the backward system \eqref{3:BS} does not depend on the control $v(\cdot)$. In this case, on the one hand, as we said earlier, the existence of MF-BSDE indicates that the backward system \eqref{3:BS} (or MF-GBCS \eqref{2:MF-GBCS}) is exactly terminal-controllable; on the other hand, the uniqueness of MF-BSDE shows that the backward system \eqref{3:BS} (or MF-GBCS \eqref{2:MF-GBCS}) is not exactly controllable.
\end{enumerate}
\end{remark}

\subsection{Gram-type criterion}

Similar to the Gram-type criteria for exact controllability of ODE systems, SDE systems (see Liu and Peng \cite{LP-10}) and MF-SDE systems (see Yu \cite{Y-21}), in this subsection, we devote ourselves to giving a Gram-type criterion for the exact controllability of MF-GBCS \eqref{2:MF-GBCS}.

Due to Corollary \ref{3.1C-3E}, we turn to consider the exact null-controllability of the backward system \eqref{3:BS}. Let us define
\begin{equation}\label{3.2:scrX}
\mathscr X := \Big\{ x^*(0; 0, v(\cdot))\, \big|\, v(\cdot) \in L^2_{\mathbb F}(0,T;\mathbb R^{m_0}) \Big\},
\end{equation}
where $\pi(\cdot;0,v(\cdot)) = ( (x^*(\cdot;0,v(\cdot))^\top, y_{-0}(\cdot;0,v(\cdot))^\top)^\top, (q(\cdot;0,v(\cdot))^\top, z_{-0}(\cdot;0,v(\cdot))^\top)^\top )^\top \in \Pi$ is the unique solution to MF-BSDE \eqref{3:BS} and \eqref{3:TC} when $x^*(T) =0$ and $v(\cdot) \in L^2_{\mathbb F}(0,T;\mathbb R^{m_0})$. Due to the linearity of \eqref{3:BS}, the set $\mathscr X \subset \mathbb R^n$ is exactly a subspace which is called the {\it null-controllable subspace of the backward system \eqref{3:BS}}. Clearly, the backward system \eqref{3:BS} is exactly null-controllable if and only if $\mathscr X = \mathbb R^n$.

Now, we introduce an $(n\times n)$ symmetrical matrix
\begin{equation}\label{3.2:G}
\begin{aligned}
\mathbf G = \begin{pmatrix} I_n & 0_{n\times(Nn)} \end{pmatrix} \mathbb E\int_0^T & \Big\{ \Phi(s) \mathbf B(s) +\mathbb E[\Phi(s)] \bar{\mathbf B}(s) \Big\}\\
\times & \Big\{ \Phi(s) \mathbf B(s) +\mathbb E[\Phi(s)] \bar{\mathbf B}(s) \Big\}^\top \, \mathrm ds \begin{pmatrix} I_n \\ 0_{(Nn) \times n}  \end{pmatrix},
\end{aligned}
\end{equation}
where $\Phi(\cdot)$ is the unique solution to the following matrix-valued linear MF-SDE:
\begin{equation}\label{3.2:Found}
\left\{
\begin{aligned}
& \mathrm d\Phi(s) = -\Big\{ \Phi(s) \mathbf A(s) +\mathbb E[\Phi(s)] \bar{\mathbf A}(s) \Big\}\, \mathrm ds\\
& \hskip 13.4mm -\Big\{ \Phi(s) \mathbf C(s) +\mathbb E[\Phi(s)] \bar{\mathbf C}(s) \Big\}\, \mathrm dW(s),\quad s\in [0,T],\\
& \Phi(0) = I_{(1+N)n}.
\end{aligned}
\right.
\end{equation}

Now we give the main result of this subsection:

\begin{theorem}\label{Gram-Cri}
Let Assumptions (H2) and (H4) hold. Then we have 
\begin{equation}\label{3.2T:Gram}
\mathscr X = \mbox{\rm Span}\ \mathbf G,
\end{equation}
where $\mathbf G$ is defined by \eqref{3.2:G} and $\mbox{\rm Span}\ \mathbf G$ represents the space spanned by all column vectors of the matrix $\mathbf G$. Consequently, the backward system \eqref{3:BS} (or MF-GBCS \eqref{2:MF-GBCS}) is exactly controllable if and only if the matrix $\mathbf G$ is non-singular.
\end{theorem}

\begin{proof}
Clearly, Equation \eqref{3.2T:Gram} is equivalent to the statement: For any $\beta\in \mathbb R^n$, $\beta\perp\mathscr X$ if and only if $\beta\perp \mathbf G$. Next we will prove that this equivalent statement hold true. 

Let $\Phi(\cdot)$ be the solution to MF-SDE \eqref{3.2:Found} and $(y(\cdot)^\top, z(\cdot)^\top)^\top = (y(\cdot;0,v(\cdot))^\top, z(\cdot;0,v(\cdot))^\top)^\top$ be the solution to MF-BSDE \eqref{3:BS} and the  supplementary terminal condition $x^*(T) =0$. Applying It\^o's formula to $\Phi(\cdot)y(\cdot)$ on the interval $[0,T]$ yields
\[
\begin{aligned}
-y(0) =\ & \mathbb E[\Phi(T)y(T)] -\Phi(0)y(0) = \mathbb E \int_0^T \Phi(s) \Big\{ \mathbf B(s)v(s) +\bar{\mathbf B}(s) \mathbb E[v(s)] \Big\}\, \mathrm ds\\
=\ & \mathbb E\int_0^T \Big\{ \Phi(s) \mathbf B(s) +\mathbb E[\Phi(s)] \bar{\mathbf B}(s) \Big\} v(s)\, \mathrm ds.
\end{aligned}
\]
Then,
\[
-x^*(0) = \begin{pmatrix} I & 0 \end{pmatrix} \mathbb E\int_0^T \Big\{ \Phi(s) \mathbf B(s) +\mathbb E[\Phi(s)] \bar{\mathbf B}(s) \Big\} v(s)\, \mathrm ds.
\]
Let $\beta\in \mathbb R^n$. We have
\begin{equation}\label{3.2T:Eq1}
-\beta^\top x^*(0) = \beta^\top \begin{pmatrix} I & 0 \end{pmatrix} \mathbb E\int_0^T \Big\{ \Phi(s) \mathbf B(s) +\mathbb E[\Phi(s)] \bar{\mathbf B}(s) \Big\} v(s)\, \mathrm ds.
\end{equation}

(Necessity). If $\beta \perp \mathscr X$, then $\beta^\top x^*(0) =0$ for any $v(\cdot)\in L^2_{\mathbb F}(0,T;\mathbb R^{m_0})$. Therefore, \eqref{3.2T:Eq1} implies
\[
\beta^\top \begin{pmatrix} I & 0 \end{pmatrix} \mathbb E\int_0^T \Big\{ \Phi(s) \mathbf B(s) +\mathbb E[\Phi(s)] \bar{\mathbf B}(s) \Big\} v(s)\, \mathrm ds =0 \quad \mbox{for any } v(\cdot) \in L^2_{\mathbb F}(0,T;\mathbb R^{m_0}).
\]
By selecting $v(\cdot)$ as each column of the matrix-valued process $\{ \Phi(\cdot) \mathbf B(\cdot) +\mathbb E[\Phi(\cdot)] \bar{\mathbf B}(\cdot) \}^\top$, we obtain
\[
\beta^\top \begin{pmatrix} I & 0 \end{pmatrix} \mathbb E\int_0^T \Big\{ \Phi(s) \mathbf B(s) +\mathbb E[\Phi(s)] \bar{\mathbf B}(s) \Big\} \Big\{ \Phi(s) \mathbf B(s) +\mathbb E[\Phi(s)] \bar{\mathbf B}(s) \Big\}^\top \, \mathrm ds =0.
\]
Consequently,
\[
\beta^\top \begin{pmatrix} I & 0 \end{pmatrix} \mathbb E\int_0^T \Big\{ \Phi(s) \mathbf B(s) +\mathbb E[\Phi(s)] \bar{\mathbf B}(s) \Big\} \Big\{ \Phi(s) \mathbf B(s) +\mathbb E[\Phi(s)] \bar{\mathbf B}(s) \Big\}^\top \, \mathrm ds \begin{pmatrix} I \\ 0 \end{pmatrix} =0,
\]
that is, $\beta^\top \mathbf G =0$ (see the definition \eqref{3.2:G} of $\mathbf G$). We have proved $\beta \perp \mathbf G$.

(Sufficiency). If $\beta \perp \mathbf G$, then $\beta^\top \mathbf G \beta =0$. By the definition \eqref{3.2:G} of $\mathbf G$, we have 
\[
\begin{aligned}
0=\ & \beta^\top \begin{pmatrix} I & 0 \end{pmatrix} \mathbb E\int_0^T \Big\{ \Phi(s) \mathbf B(s) +\mathbb E[\Phi(s)] \bar{\mathbf B}(s) \Big\} \Big\{ \Phi(s) \mathbf B(s) +\mathbb E[\Phi(s)] \bar{\mathbf B}(s) \Big\}^\top \, \mathrm ds \begin{pmatrix} I \\ 0 \end{pmatrix} \beta\\
=\ & \mathbb E\int_0^T \Big| \beta^\top \begin{pmatrix} I & 0 \end{pmatrix} \Big\{ \Phi(s) \mathbf B(s) +\mathbb E[\Phi(s)] \bar{\mathbf B}(s) \Big\} \Big|^2\, \mathrm ds.
\end{aligned}
\]
Therefore,
\[
\beta^\top \begin{pmatrix} I & 0 \end{pmatrix} \Big\{ \Phi(s) \mathbf B(s) +\mathbb E[\Phi(s)] \bar{\mathbf B}(s) \Big\} =0 \quad \mbox{for almost all } (s,\omega) \in [0,T]\times \Omega.
\]
Substituting the above equation into \eqref{3.2T:Eq1} leads to $\beta^\top x^*(0) =0$. From the arbitrariness of $v(\cdot) \in L^2_{\mathbb F}(0,T;\mathbb R^{m_0})$, we get $\beta \perp \mathscr X$. The proof is completed.
\end{proof}

%\begin{corollary}\label{3.2C:Ito}
%Let Assumptions (H2) and (H4) hold. Let $\bar{\mathbf A}(\cdot) =\bar{\mathbf C}(\cdot) =0$ and $\bar{\mathbf B}(\cdot) =0$ (which can be obviously implied by $\bar A(\cdot) = \bar C(\cdot) =0$, $\bar B(\cdot) =\bar D(\cdot) =0$, $\bar Q_i(\cdot) =0$ and $\bar R_i(\cdot) =0$ ($i=1,2,\dots,N$)). The backward system \eqref{3:BS} (or MF-GBCS \eqref{2:MF-GBCS}) is exactly controllable if and only if the matrix
%\[
%\mathbf G = \begin{pmatrix} I_n & 0_{n\times(Nn)} \end{pmatrix} \mathbb E\int_0^T \Phi(s) \mathbf B(s) \mathbf B(s)^\top \Phi(s)^\top \, \mathrm ds \begin{pmatrix} I_n \\ 0_{(Nn) \times n}  \end{pmatrix}
%\]
%is non-singular, where $\Phi(\cdot)$ is the unique solution to the following linear SDE:
%\[
%\left\{
%\begin{aligned}
%& \mathrm d\Phi(s) = -\Phi(s) \mathbf A(s)\, \mathrm ds -\Phi(s) \mathbf C(s)\, \mathrm dW(s),\quad s\in [0,T],\\
%& \Phi(0) = I_{(1+N)n}.
%\end{aligned}
%\right.
%\]
%We notice that, in this case, there is no mean-field items involved. The result in this corollary is exactly the result of Theorem 4.3 in Liu and Yu \cite{LY-22+}.
%\end{corollary}

\subsection{Kalman-type criterion for time-invariant coefficients}

This subsection is concerned with a special case where the involved coefficients are time-invariant. In detail, we introduce

\medskip

\noindent{\bf Assumption (H5).} The coefficients $A(\cdot)$, $\bar A(\cdot)$, $B(\cdot)$, $\bar B(\cdot)$, $C(\cdot)$, $\bar C(\cdot)$, $D(\cdot)$, $\bar D(\cdot)$ appearing in \eqref{2:Sys} and $Q_i(\cdot)$, $\bar Q_i(\cdot)$, $R_i(\cdot)$, $\bar R_i(\cdot)$ ($i=1,2,\dots,N$) appearing in \eqref{2:Cost} are time-invariant.

\medskip

\noindent Consequently, in this special case, the coefficients $\mathbf A(\cdot)$, $\bar{\mathbf A}(\cdot)$, $\mathbf B(\cdot)$, $\bar{\mathbf B}(\cdot)$, $\mathbf C(\cdot)$ and $\bar{\mathbf C}(\cdot)$ in the backward system \eqref{3:BS} are also time-invariant.

Under Assumption (H5), in this subsection we will obtain another simpler criterion for the exact controllability of the backward system \eqref{3:BS} (or MF-GBCS \eqref{2:MF-GBCS}), named Kalman-type criterion. The proof of the forthcoming Kalman-type criterion depends mainly on a result in Yu \cite{Y-21}. Next we will state this preliminary result and some relevant notations.

Firstly, in \cite{Y-21}, the author combined the classical product and the tensor product of matrices to define the following new multiplication operation:

\begin{definition}[Definition 5.1 in \cite{Y-21}]
For any given block matrices
\[
\mathbb A = \left( \begin{array}{cccc}
\mathbb A_{11} & \mathbb A_{12} & \cdots & \mathbb A_{1n}\\
\mathbb A_{21} & \mathbb A_{22} & \cdots & \mathbb A_{2n}\\
\vdots & \vdots & \ddots & \vdots\\
\mathbb A_{m1} & \mathbb A_{m2} & \cdots & \mathbb A_{mn}
\end{array} \right) \quad \mbox{and} \quad
\mathbb B = \left( \begin{array}{cccc}
\mathbb B_{11} & \mathbb B_{12} & \cdots & \mathbb B_{1q}\\
\mathbb B_{21} & \mathbb B_{22} & \cdots & \mathbb B_{2q}\\
\vdots & \vdots & \ddots & \vdots\\
\mathbb B_{p1} & \mathbb B_{p2} & \cdots & \mathbb B_{pq}
\end{array} \right),
\]
their {\rm block-tensor product} is defined as follows:
\[
\mathbb A \otimes \mathbb B = \left(\begin{array}{cccccccccc} 
\mathbb A_{11}\mathbb B_{11} & \mathbb A_{11}\mathbb B_{12} & \cdots & \mathbb A_{11}\mathbb B_{1q} & \cdots & \cdots & \mathbb A_{1n}\mathbb B_{11} & \mathbb A_{1n}\mathbb B_{12} & \cdots & \mathbb A_{1n}\mathbb B_{1q}\\
\mathbb A_{11}\mathbb B_{21} & \mathbb A_{11}\mathbb B_{22} & \cdots & \mathbb A_{11}\mathbb B_{2q} & \cdots & \cdots & \mathbb A_{1n}\mathbb B_{21} & \mathbb A_{1n}\mathbb B_{22} & \cdots & \mathbb A_{1n}\mathbb B_{2q}\\
\vdots & \vdots & \ddots & \vdots & & & \vdots & \vdots & \ddots & \vdots\\
\mathbb A_{11}\mathbb B_{p1} & \mathbb A_{11}\mathbb B_{p2} & \cdots & \mathbb A_{11}\mathbb B_{pq} & \cdots & \cdots & \mathbb A_{1n}\mathbb B_{p1} & \mathbb A_{1n}\mathbb B_{p2} & \cdots & \mathbb A_{1n}\mathbb B_{pq}\\
\vdots & \vdots & & \vdots & \ddots &  & \vdots & \vdots & & \vdots\\
\vdots & \vdots & & \vdots &  & \ddots & \vdots & \vdots & & \vdots\\
\mathbb A_{m1}\mathbb B_{11} & \mathbb A_{m1}\mathbb B_{12} & \cdots & \mathbb A_{m1}\mathbb B_{1q} & \cdots & \cdots & \mathbb A_{mn}\mathbb B_{11} & \mathbb A_{mn}\mathbb B_{12} & \cdots & \mathbb A_{mn}\mathbb B_{1q}\\
\mathbb A_{m1}\mathbb B_{21} & \mathbb A_{m1}\mathbb B_{22} & \cdots & \mathbb A_{m1}\mathbb B_{2q} & \cdots & \cdots & \mathbb A_{mn}\mathbb B_{21} & \mathbb A_{mn}\mathbb B_{22} & \cdots & \mathbb A_{mn}\mathbb B_{2q}\\
\vdots & \vdots & \ddots & \vdots & & & \vdots & \vdots & \ddots & \vdots\\
\mathbb A_{m1}\mathbb B_{p1} & \mathbb A_{m1}\mathbb B_{p2} & \cdots & \mathbb A_{m1}\mathbb B_{pq} & \cdots & \cdots & \mathbb A_{mn}\mathbb B_{p1} & \mathbb A_{mn}\mathbb B_{p2} & \cdots & \mathbb A_{mn}\mathbb B_{pq}
\end{array}\right)
\]
provided all the involved classical products $\mathbb A_{ij} \mathbb B_{hl}$ are well-posed ($i=1,2,\dots,m$; $j=1,2,\dots,n$; $h=1,2,\dots,p$; $l=1,2,\dots, q$). 
\end{definition}

\noindent The above defined block-tensor product obeys the following combination rule: 
\[
(\mathbb A \otimes \mathbb B) \otimes \mathbb C = \mathbb A \otimes (\mathbb B \otimes \mathbb C) =: \mathbb A \otimes \mathbb B \otimes \mathbb C.
\]
Consequently, the following notation is unambiguous:
\[
\overbrace{\mathbb A \otimes \mathbb A \otimes \cdots \otimes \mathbb A}^{\text{the number is } k} =: \mathbb A^{\otimes k}.
\]

Secondly, based on the coefficients $\mathbf A$, $\bar{\mathbf A}$, $\mathbf B$, $\bar{\mathbf B}$, $\mathbf C$ and $\bar{\mathbf C}$ in the backward system \eqref{3:BS}, we define 
\begin{equation}\label{3.3:D1D2}
\mathcal D_1 = \widehat{\mathbf B}, \qquad \mathcal D_2 = \begin{pmatrix} \widehat{\mathbf A}\widehat{\mathbf B}, & \widehat{\mathbf C} \mathbf B \end{pmatrix}
\end{equation}
and
\begin{equation}\label{3.3:Dk}
\mathcal D_k = \left\{
\begin{aligned}
& \begin{pmatrix} \widehat{\mathbf A} \otimes \mathcal D_{k-1}, & \widehat{\mathbf C} \otimes \begin{pmatrix} \mathbf A, & \mathbf C \end{pmatrix}^{\otimes (k-2)} \otimes \mathbf B \end{pmatrix}, \quad k=3,4,\dots,(1+N)n+1,\\
& \left( \begin{matrix} 
\widehat{\mathbf A}^{k-1} \widehat{\mathbf B}, &  
\widehat{\mathbf A}^{k-2} \widehat{\mathbf C} \mathbf B, &
\widehat{\mathbf A}^{k-3} \widehat{\mathbf C} \otimes \begin{pmatrix} \mathbf A, & \mathbf C \end{pmatrix}\otimes \mathbf B, & \cdots,
\end{matrix} \right.\\
& \left. \begin{matrix}
\quad \widehat{\mathbf A}^{k-(1+N)n-1} \widehat{\mathbf C} \otimes \begin{pmatrix} \mathbf A, & \mathbf C \end{pmatrix}^{\otimes ((1+N)n-1)}\otimes \mathbf B
\end{matrix} \right), \quad k=(1+N)n+2,\dots,2(1+N)n.
\end{aligned}
\right.
\end{equation}

Thirdly, similar to \eqref{3.2:scrX}, we also introduce
\begin{equation}\label{3.3:scrY}
\mathscr Y := \Big\{ y(0; 0, v(\cdot))\, \big|\, v(\cdot) \in L^2_{\mathbb F}(0,T;\mathbb R^{m_0}) \Big\},
\end{equation}
where $\pi(\cdot;0,v(\cdot)) = ( y(\cdot;0,v(\cdot))^\top, z(\cdot;0,v(\cdot))^\top )^\top \in \Pi$ is the unique solution to MF-BSDE \eqref{3:BS} and \eqref{3:TC} when $x^*(T) =0$ and $v(\cdot) \in L^2_{\mathbb F}(0,T;\mathbb R^{m_0})$. Clearly, $\mathscr Y$ is a subspace of $\mathbb R^{(1+N)n}$.

\begin{lemma}[Theorem 5.9 in \cite{Y-21}]\label{3.3Lem-Yu}
Under Assumptions (H2), (H4) and (H5), we have
\begin{equation}
\mathscr Y = \mbox{\rm Span}\ \begin{pmatrix} \mathcal D_1 & \mathcal D_2 & \cdots & \mathcal D_{2(1+N)n} \end{pmatrix},
\end{equation}
where $\mathscr Y$ and $\mathcal D_k$ ($k=1,2,\dots,2(1+N)n$) are defined by \eqref{3.3:scrY}, \eqref{3.3:D1D2} and \eqref{3.3:Dk}, respectively.
\end{lemma}

We are in the position to give the main result of this subsection:

\begin{theorem}\label{Kalman-Cri}
Let Assumptions (H2), (H4) and (H5) hold. Then we have 
\begin{equation}\label{3.3T:Kal}
\mathscr X = \mbox{\rm Span}\ \Big\{ \begin{pmatrix} I_n & 0_{n\times (Nn)} \end{pmatrix} 
\begin{pmatrix} \mathcal D_1 & \mathcal D_2 & \cdots & \mathcal D_{2(1+N)n} \end{pmatrix} \Big\},
\end{equation}
where $\mathcal D_k$ ($k=1,2,\dots,2(1+N)n$) is defined by \eqref{3.3:D1D2} and \eqref{3.3:Dk}. Consequently, the backward system \eqref{3:BS} (or MF-GBCS \eqref{2:MF-GBCS}) is exactly controllable if and only if 
\begin{equation}
\mbox{\rm Rank}\ \Big\{ \begin{pmatrix} I_n & 0_{n\times (Nn)} \end{pmatrix} 
\begin{pmatrix} \mathcal D_1 & \mathcal D_2 & \cdots & \mathcal D_{2(1+N)n} \end{pmatrix} \Big\} =n.
\end{equation}
\end{theorem}

\begin{proof}
We notice that $x^*(0;0,v(\cdot)) = \begin{pmatrix} I_n & 0_{n\times (Nn)} \end{pmatrix} y(0;0,v(\cdot))$ holds for any $v(\cdot) \in L^2_{\mathbb F}(0,T;\mathbb R^{m_0})$. Then, \eqref{3.3T:Kal} is directly derived from Lemma \ref{3.3Lem-Yu}.
\end{proof}

Theorem \ref{Kalman-Cri} provides a rank condition for the exact controllability of the backward system \eqref{3:BS} (or MF-GBCS \eqref{2:MF-GBCS}), which we call Kalman-type criterion. Next, we give an example to show the application of the Kalman-type criterion.

\begin{example}
We assume that there are two agents at the lower layer, i.e. $N=2$. Let the dimension of the state $n=2$, the dimensions of Agents' controls $m_1=m_2=1$ and the dimension of the regulator at the upper layer $m_0 =3$. The related coefficients in the system \eqref{2:Sys} and the cost functionals \eqref{2:Cost} are set as follows:
\[
\begin{aligned}
& C = -A = I_2,\quad 
D_1 = \bar D_2 = 2B_1 = 2\bar B_2 = \begin{pmatrix} 2 \\ 0 \end{pmatrix},\quad 
\bar D_1 = D_2 = 2\bar B_1 = 2B_2 = \begin{pmatrix} 0 \\ 2 \end{pmatrix},\\
& \bar C = \bar A = 0_{2\times 2},\quad
\bar D_0 =\bar B_0 = 0_{2\times 3},\quad
D_0 = \begin{pmatrix} 1 & 1 & 0 \\ 1 & 1 & 1 \end{pmatrix},\quad
B_0 = \begin{pmatrix} 1 & 1 & 0 \\ 1 & 0 & 1 \end{pmatrix},\\
& \bar Q_2 = 2Q_1 = \begin{pmatrix} 2 & 0 \\ 0 & 0 \end{pmatrix},\quad
Q_2 =2\bar Q_1 = \begin{pmatrix} 0 & 0 \\ 0 & 2 \end{pmatrix},\quad 
R_2 =\bar R_2 =2R_1 =2\bar R_1 =2.
\end{aligned}
\]
Since the terminal coefficients $H_1$, $\bar H_1$, $H_2$ and $\bar H_2$ do not affect the exact controllability (see Remark \ref{3.1:RMK}-(i)), then we do not give their specific setting here. Based on the above setting, we calculate the related coefficients in the backward system \eqref{3:BS}:
\[
\begin{aligned}
& \mathbf A = \begin{pmatrix} 
\begin{matrix} -2 & 0 \\ \frac 1 2 & -2 \end{matrix} &
\begin{matrix} 1 & 0 & 0 & 0 \\ -1 & 0 & 0 & \frac 1 2 \end{matrix}\\
\begin{matrix} -1 & 0 \\ 0 & 0 \\ 0 & 0 \\ 0 & -2 \end{matrix} & 
I_4
\end{pmatrix},\quad
\bar{\mathbf A} = \begin{pmatrix}
0_{2\times 2} & 
\begin{matrix} -\frac 1 2 & \frac 1 2 & \frac 1 4 & \frac 1 4 \\
1 & 0 & 0 & -\frac 1 2 \end{matrix}\\
\begin{matrix} 0 & 0 \\ 0 & -1 \\ -2 & 0 \\ 0 & 0 \end{matrix} &
0_{4\times 4}
\end{pmatrix},\quad 
\bar{\mathbf B} = 0_{6\times 3},\\
& \mathbf C = \begin{pmatrix}
\begin{matrix} 1 & 0 \\ -\frac 1 2 & 1 \end{matrix} &
\begin{matrix} 2 & 0 & 0 & 0 \\ -2 & 0 & 0 & 1 \end{matrix} \\
0_{4\times 2} & -I_4
\end{pmatrix},\quad
\bar{\mathbf C} = \begin{pmatrix}
\begin{matrix} 0_{2\times 2} & 
\begin{matrix} -1 & 1 & \frac 1 2 & \frac 1 2 \\ 2 & 0 & 0 & -1 \end{matrix} \end{matrix}\\
0_{4\times 6}
\end{pmatrix},\quad
\mathbf B = \begin{pmatrix}
\begin{matrix} 0 & 0 & 0 \\ \frac 1 2 & -\frac 1 2 & 0 \end{matrix} \\
0_{4\times 3}
\end{pmatrix}.
\end{aligned}
\]

Due to Theorem \ref{Kalman-Cri}, we need to verify whether $\mbox{Rank}\ \{ \begin{pmatrix} I_2 & 0_{2\times 4} \end{pmatrix} 
\begin{pmatrix} \mathcal D_1 & \mathcal D_2 & \cdots & \mathcal D_{12} \end{pmatrix} \} =2$ is true. For this aim, we calculate
\[
\mathcal D_3 = \begin{pmatrix}  
\widehat{\mathbf A} \widehat{\mathbf A} \widehat{\mathbf B}, &
\widehat{\mathbf A} \widehat{\mathbf C} \mathbf B, &
\widehat{\mathbf C} \mathbf{AB}, &
\widehat{\mathbf C} \mathbf{CB}
\end{pmatrix}
\]
and
\[
\begin{pmatrix} I_2 & 0_{2\times 4} \end{pmatrix} \widehat{\mathbf A} \widehat{\mathbf A} \widehat{\mathbf B} 
= \begin{pmatrix} -\frac 1 2 & \frac 1 2 & 0 \\ 2 & -2 & 0 \end{pmatrix},\quad 
\begin{pmatrix} I_2 & 0_{2\times 4} \end{pmatrix}\widehat{\mathbf A} \widehat{\mathbf C} \mathbf B 
= \begin{pmatrix} 0 & 0 & 0 \\ -1 & 1 & 0 \end{pmatrix}.
\]
Then,
\[
\begin{aligned}
2 \geq\ & \mbox{Rank}\ \Big\{ \begin{pmatrix} I_2 & 0_{2\times 4} \end{pmatrix} 
\begin{pmatrix} \mathcal D_1 & \mathcal D_2 & \cdots & \mathcal D_{12} \end{pmatrix} \Big\}\\
\geq\ & \mbox{Rank}\ \Big\{ \begin{pmatrix} I_2 & 0_{2\times 4} \end{pmatrix} 
\begin{pmatrix} \widehat{\mathbf A} \widehat{\mathbf A} \widehat{\mathbf B}, &
\widehat{\mathbf A} \widehat{\mathbf C} \mathbf B \end{pmatrix} \Big\}
=2.
\end{aligned}
\]
Therefore, the MF-GBCS considered in this example is exactly controllable.
\end{example}

We further introduce

\medskip

\noindent{\bf Assumption (H6).} The coefficients $\bar A(\cdot) =\bar C(\cdot) =0$, $\bar B(\cdot) = \bar D(\cdot) =0$, $\bar Q_i(\cdot) =0$ and $\bar R_i(\cdot) =0$ ($i=1,2,\dots,N$).

\medskip

\noindent Clearly, Assumption (H6) implies that $\bar{\mathbf A}(\cdot) =\bar{\mathbf C}(\cdot) =0$ and $\bar{\mathbf B}(\cdot) =0$. We notice that, on the issue of exact controllability, Assumption (H6) means that no mean-field item is involved in the system (see Remark \ref{3.1:RMK}-(i)). The following corollary improved the result of Theorem 3.14 in Zhang and Guo \cite{ZhG-19SICON}.

\begin{corollary}
Let Assumptions (H2), (H4), (H5) and (H6) hold. Then we have
\begin{equation}
\mathscr X = \mbox{\rm Span}\ \Big\{ \begin{pmatrix} I_n & 0_{n\times (Nn)} \end{pmatrix} 
\Big( \begin{matrix} \mathbf B, & \begin{pmatrix} \mathbf A, & \mathbf C \end{pmatrix} \otimes \mathbf B, & \dots, & \begin{pmatrix} \mathbf A, & \mathbf C \end{pmatrix}^{\otimes ((1+N)n -1)} \otimes \mathbf B \end{matrix} \Big) \Big\}.
\end{equation} 
Consequently, the backward system \eqref{3:BS} (or MF-GBCS \eqref{2:MF-GBCS}) is exactly controllable if and only if 
\begin{equation}
\mbox{\rm Rank}\ \Big\{ \begin{pmatrix} I_n & 0_{n\times (Nn)} \end{pmatrix} 
\Big( \begin{matrix} \mathbf B, & \begin{pmatrix} \mathbf A, & \mathbf C \end{pmatrix} \otimes \mathbf B, & \dots, & \begin{pmatrix} \mathbf A, & \mathbf C \end{pmatrix}^{\otimes ((1+N)n -1)} \otimes \mathbf B \end{matrix} \Big) \Big\} =n.
\end{equation}
\end{corollary}

\begin{proof}
By some careful calculations under Assumption (H6), we have
\[
\mathcal D_k = \begin{pmatrix} \mathbf A, & \mathbf C \end{pmatrix}^{\otimes (k-1)} \otimes \mathbf B, \qquad k=1,2,\dots,(1+N)n+1
\]
and
\[
\mbox{Span}\ \mathcal D_k \subset \mbox{Span}\ \Big\{ \begin{pmatrix} \mathbf A, & \mathbf C \end{pmatrix}^{\otimes (k-1)} \otimes \mathbf B \Big\}, \qquad k=(1+N)n+2,\dots,2(1+N)n
\]
(see \eqref{3.3:D1D2} and \eqref{3.3:Dk} in this subsection and (5.5) in \cite{Y-21}). We notice that, Theorem 6.1 in L\"u and Zhang \cite{LZh-BOOK} (see also Lemma 5.8 in \cite{Y-21} restated by the block-tensor product of matrices) shows 
\[
\begin{aligned}
& \mbox{Span}\ \Big( \begin{matrix} \mathbf B, & \begin{pmatrix} \mathbf A, & \mathbf C \end{pmatrix} \otimes \mathbf B, & \dots, & \begin{pmatrix} \mathbf A, & \mathbf C \end{pmatrix}^{\otimes k} \otimes \mathbf B, & \dots \end{matrix} \Big)\\
=\ & \mbox{Span}\ \Big( \begin{matrix} \mathbf B, & \begin{pmatrix} \mathbf A, & \mathbf C \end{pmatrix} \otimes \mathbf B, & \dots, & \begin{pmatrix} \mathbf A, & \mathbf C \end{pmatrix}^{\otimes ((1+N)n-1)} \otimes \mathbf B \end{matrix} \Big).
\end{aligned}
\]
Then,
\[
\begin{aligned}
& \mbox{Span}\ \begin{pmatrix} \mathcal D_1 & \mathcal D_2 & \cdots & \mathcal D_{2(1+N)n} \end{pmatrix}\\
=\ & \mbox{Span}\ \Big( \begin{matrix} \mathbf B, & \begin{pmatrix} \mathbf A, & \mathbf C \end{pmatrix} \otimes \mathbf B, & \dots, & \begin{pmatrix} \mathbf A, & \mathbf C \end{pmatrix}^{\otimes ((1+N)n-1)} \otimes \mathbf B \end{matrix} \Big).
\end{aligned}
\]
Consequently,
\[
\begin{aligned}
& \mbox{Span}\ \Big\{ \begin{pmatrix} I_n & 0_{n\times (Nn)} \end{pmatrix} \begin{pmatrix} \mathcal D_1 & \mathcal D_2 & \cdots & \mathcal D_{2(1+N)n} \end{pmatrix} \Big\}\\
=\ & \mbox{Span}\ \Big\{ \begin{pmatrix} I_n & 0_{n\times (Nn)} \end{pmatrix} \Big( \begin{matrix} \mathbf B, & \begin{pmatrix} \mathbf A, & \mathbf C \end{pmatrix} \otimes \mathbf B, & \dots, & \begin{pmatrix} \mathbf A, & \mathbf C \end{pmatrix}^{\otimes ((1+N)n-1)} \otimes \mathbf B \end{matrix} \Big)\Big\}.
\end{aligned}
\]
Finally, Theorem \ref{Kalman-Cri} works to finish the proof.
\end{proof}

\subsection{Exact observability of dual system}

In this subsection, we will adopt the dual point of view to provide an equivalent way to judge the exact controllability of the system \eqref{3:BS} and $x^*(T) =0$, i.e. to judge the exact observability of a certain dual system.

As the beginning, let us consider the solution of MF-BSDE \eqref{3:BS} and \eqref{3:TC}. We introduce two operators $\mathbb K: L^2_{\mathbb F}(0,T;\mathbb R^{m_0}) \rightarrow \mathbb R^n$ and $\mathbb L: L^2_{\mathcal F_T}(\Omega;\mathbb R^n) \rightarrow \mathbb R^n$ to separate the influence of $v(\cdot)$ and $x_T$ as follows:
\begin{equation}\label{3.4:KL}
\left\{
\begin{aligned}
& \mathbb Kv(\cdot) = x^*(0;0,v(\cdot)) \quad \mbox{for any } v(\cdot) \in L^2_{\mathbb F}(0,T;\mathbb R^{m_0}),\\
& \mathbb Lx_T = x^*(0;x_T,0) \quad \mbox{for any } x_T\in L^2_{\mathcal F_T}(\Omega;\mathbb R^n),
\end{aligned}
\right.
\end{equation}
where $\pi(\cdot;x_T,v(\cdot)) = ( (x^*(\cdot;x_T,v(\cdot))^\top, y_{-0}(\cdot;x_T,v(\cdot))^\top)^\top, (q(\cdot;x_T,v(\cdot))^\top, z_{-0}(\cdot;x_T,v(\cdot))^\top)^\top )^\top \in \Pi$ is the unique solution to MF-BSDE \eqref{3:BS} and \eqref{3:TC} when $x^*(T) =x_T$ and $v(\cdot) \in L^2_{\mathbb F}(0,T;\mathbb R^{m_0})$. By Definition \eqref{3.2:scrX}, we know $\mathscr X = \mathbb K(L^2_{\mathbb F}(0,T;\mathbb R^{m_0}))$, the range of $\mathbb K$. Then the backward system \eqref{3:BS} is exactly controllable if and only if $\mathbb K$ is surjective. From the linearity of MF-BSDE \eqref{3:BS} and \eqref{3:TC}, we have
\[
\pi(\cdot;x_T,v(\cdot)) = \pi(\cdot;0,v(\cdot)) +\pi(\cdot;x_T,0) \ \mbox{for any } x_T\in L^2_{\mathcal F_T}(\Omega;\mathbb R^n) \mbox{ and any } v(\cdot) \in L^2_{\mathbb F}(0,T;\mathbb R^{m_0}).
\]
Especially,
\[
x^*(0;x_T,v(\cdot)) = \mathbb K v(\cdot) +\mathbb L x_T \quad \mbox{for any } x_T\in L^2_{\mathcal F_T}(\Omega;\mathbb R^n) \mbox{ and any } v(\cdot) \in L^2_{\mathbb F}(0,T;\mathbb R^{m_0}).
\]
The linearity of MF-BSDE \eqref{3:BS} and \eqref{3:TC} also implies that both $\mathbb K$ and $\mathbb L$ are linear operators. Moreover, the continuous dependence of solutions of MF-BSDEs (see \cite[Proposition 2.2]{TY-23} for example) implies that both $\mathbb K$ and $\mathbb L$ are bounded.

Now we introduce the following system: 
\begin{equation}\label{3.4:DS}
\left\{
\begin{aligned}
& \mathrm d\xi(s) = -\Big\{ \mathbf A(s)^\top \xi(s) +\bar{\mathbf A}(s)^\top \mathbb E[\xi(s)] \Big\}\, \mathrm ds\\
& \hskip 12.6mm -\Big\{ \mathbf C(s)^\top \xi(s) +\bar{\mathbf C}(s)^\top \mathbb E[\xi(s)] \Big\}\, \mathrm dW(s),\quad s\in [0,T],\\
& \xi_{-0} =0,
\end{aligned}
\right.
\end{equation}
where the stochastic process $\xi(\cdot) := ( \eta(\cdot)^\top, \xi_{-0}(\cdot)^\top )^\top$ taking values in $\mathbb R^n \times \mathbb R^{Nn}$. Let us supplement an initial condition
\begin{equation}\label{3.4:IC}
\eta(0) =\eta_0 \in \mathbb R^n.
\end{equation}
Then, \eqref{3.4:DS} and \eqref{3.4:IC} form a linear MF-SDE which admits a unique solution $\xi(\cdot) \equiv \xi(\cdot;\eta_0) \in L^2_{\mathbb F}(\Omega;C(0,T;\mathbb R^{(1+N)n}))$. Let $\pi(\cdot) =\pi(\cdot;x_T,v(\cdot))$ denote the unique solution to MF-BSDE \eqref{3:BS} and \eqref{3:TC}. Applying It\^o's formula to $\langle \xi(\cdot),\ y(\cdot) \rangle \equiv \langle \xi(\cdot;\eta_0),\ y(\cdot;x_T,v(\cdot)) \rangle$ yields
\begin{equation}\label{3.4:Dual}
\begin{aligned}
\big\langle \eta_0,\ x^*(0;x_T,v(\cdot)) \big\rangle
=\ & \mathbb E\Big\langle \eta(T;\eta_0) +H^\top \xi_{-0}(T;\eta_0) +\bar H^\top \mathbb E\big[ \xi_{-0}(T;\eta_0) \big],\ x_T \Big\rangle \\
& -\mathbb E\int_0^T \Big\langle \mathbf B(s)^\top\xi(s;\eta_0) +\bar{\mathbf B}(s)^\top \mathbb E\big[ \xi(s;\eta_0) \big],\ v(s) \Big\rangle\, \mathrm ds.
\end{aligned}
\end{equation}
Denote by $\mathbb K^*: \mathbb R^n \rightarrow L^2_{\mathbb F}(0,T;\mathbb R^n)$ and $\mathbb L^*: \mathbb R^n \rightarrow L^2_{\mathcal F_T}(\Omega;\mathbb R^n)$ the adjoint operators of $\mathbb K$ and $\mathbb L$ (see the definition \eqref{3.4:KL}), respectively. Then Equation \eqref{3.4:Dual} shows that
\begin{equation}\label{3.4:K*}
\mathbb K^*\eta_0 = -\mathbf B(\cdot)^\top \xi(\cdot;\eta_0) -\bar{\mathbf B}(\cdot)^\top \mathbb E \big[ \xi(\cdot;\eta_0) \big] \quad \mbox{for any } \eta_0 \in \mathbb R^n
\end{equation}
and
\begin{equation}\label{3.4:L*}
\mathbb L^*\eta_0 = \eta(T;\eta_0) +H^\top \xi_{-0}(T;\eta_0) +\bar H^\top \mathbb E\big[ \xi_{-0}(T;\eta_0) \big] \quad \mbox{for any } \eta_0 \in \mathbb R^n.
\end{equation}

\begin{definition}
\begin{enumerate}[(i).]
\item We call $\mathbb K^*$ defined by \eqref{3.4:K*} an {\rm observer} of the system \eqref{3.4:DS} on $[0,T]$.
\item The system \eqref{3.4:DS} together with the observer \eqref{3.4:K*} is said to be {\rm exactly observable} on $[0,T]$ if the initial value $\eta_0\in\mathbb R^n$ of $\eta(\cdot)$ can be uniquely determined from the {\rm observation} $\mathbb K^*\eta_0$, i.e. the operator $\mathbb K^*$ is injective.
\end{enumerate}
\end{definition}

\begin{remark}
The exact observability and the exact controllability are called a pair of dual concepts. The observation system \eqref{3.4:DS} and \eqref{3.4:K*} and the control system \eqref{3:BS} and $x^*(T) =0$ are called a pair of dual systems. 
\end{remark}

The following estimate of MF-SDE \eqref{3.4:DS} and \eqref{3.4:IC} is often used to characterize the exact observability of the system \eqref{3.4:DS} and \eqref{3.4:K*}.

\begin{proposition}\label{3.4P:OI}
Let Assumptions (H2) and (H4) hold. The system \eqref{3.4:DS} and \eqref{3.4:K*} is exactly observable if and only if, there exists a constant $\delta>0$ such that
\begin{equation}\label{3.4:OI}
\Vert \mathbb K^*\eta_0 \Vert_{L^2_{\mathbb F}(0,T;\mathbb R^{m_0})}^2 = \mathbb E\int_0^T \Big| \mathbf B(s)^\top \xi(s;\eta_0) +\bar{\mathbf B}(s)^\top \mathbb E \big[ \xi(s;\eta_0) \big] \Big|^2\, \mathrm ds \geq \delta |\eta_0|^2 \quad \mbox{for any } \eta_0\in \mathbb R^n.
\end{equation}
The above inequality \eqref{3.4:OI} is called the {\rm observability inequality} of the system \eqref{3.4:DS} and \eqref{3.4:K*}.
\end{proposition}

\begin{proof}
The sufficiency is obvious, and we only need to prove the necessity. If the system \eqref{3.4:DS} and \eqref{3.4:K*} is exactly observable, then the bounded linear operator $\mathbb K^* : \mathbb R^n \rightarrow \mathbb K^*(\mathbb R^n) \subset L^2_{\mathbb F}(0,T;\mathbb R^{m_0})$ is bijective. Because $\mathbb R^n$ is finite-dimensional, so is $\mathbb K^*(\mathbb R^n)$. Consequently, $\mathbb K^*(\mathbb R^n)$ is a complete linear normed space. By Banach's inverse operator theorem, $(\mathbb K^*)^{-1}: \mathbb K^*(\mathbb R^n) \rightarrow \mathbb R^n$ is bounded also. Therefore, the observability inequality \eqref{3.4:OI} holds.
\end{proof}

In order to better understand the observability inequality \eqref{3.4:OI}, let us analyze the structure of the solution $\xi(\cdot;\eta_0)$ to MF-SDE \eqref{3.4:DS} and \eqref{3.4:IC}. With the help of the matrix-valued MF-SDE \eqref{3.2:Found} and the linearity, we have
\begin{equation}\label{3.4:xiPhi}
\xi(\cdot;\eta_0) = \Phi(\cdot)^\top 
\begin{pmatrix} I_n \\ 0_{(Nn)\times n}\end{pmatrix}\eta_0.
\end{equation}
Then,
\begin{equation}\label{3.4:K*G}
\begin{aligned}
\Vert \mathbb K^*\eta_0 \Vert_{L^2_{\mathbb F}(0,T;\mathbb R^{m_0})}^2
=\ & \mathbb E\int_0^T \Big| \mathbf B(s)^\top \Phi(s)^\top \begin{pmatrix} I \\ 0 \end{pmatrix} \eta_0 +\bar{\mathbf B}(s)^\top \mathbb E\big[ \Phi(s)^\top \big] \begin{pmatrix} I \\ 0 \end{pmatrix} \eta_0 \Big|^2\, \mathrm ds\\
=\ & \mathbb E\int_0^T \Big| \Big\{ \mathbf B(s)^\top \Phi(s)^\top +\bar{\mathbf B}(s)^\top \mathbb E\big[ \Phi(s)^\top \big] \Big\} \begin{pmatrix} I \\ 0 \end{pmatrix} \eta_0 \Big|^2\, \mathrm ds\\
=\ & \langle \mathbf G\eta_0,\ \eta_0 \rangle,
\end{aligned}
\end{equation}
where $\mathbf G$ is the Gramian matrix defined by \eqref{3.2:G}. Consequently, the observability inequality \eqref{3.4:OI} holds if and only if the matrix $\mathbf G$ is non-singular.

The above analysis, Proposition \ref{3.4P:OI} and Theorem \ref{Gram-Cri} implies the following

\begin{theorem}\label{3.4T:ECEO}
Let Assumptions (H2) and (H4) hold. Then the backward system \eqref{3:BS} (or MF-GBCS \eqref{2:MF-GBCS}) is exactly controllable if and only if the system \eqref{3.4:DS} and \eqref{3.4:K*} is exactly observable.
\end{theorem}

\section{Regulator's control steering the state from $x_0$ to $x_T$}\label{SEC:Steer}

In the previous Section \ref{SEC:Exact}, we give three approaches to judge the exact controllability of MF-GBCS \eqref{2:MF-GBCS}: the Gram-type criterion (see Theorem \ref{Gram-Cri}), the Kalman-type criterion for time-invariant coefficients (see Theorem \ref{Kalman-Cri}) and the exact observability of the dual system \eqref{3.4:DS} and \eqref{3.4:K*} (see Theorem \ref{3.4T:ECEO}) with the help of the equivalent backward system \eqref{3:BS}. However, it is not provided how to select some regulator's admissible control to steer the state $x^*(\cdot)$ from an arbitrary given initial value $x_0\in \mathbb R^n$ to an arbitrary given terminal target $x_T\in L^2_{\mathcal F_t}(\Omega;\mathbb R^n)$ when MF-GBCS \eqref{2:MF-GBCS} is exactly controllable. In this section, we aim to solve this problem.

For any $(x_0,x_T) \in \mathbb R^n \times L^2_{\mathcal F_T}(\Omega;\mathbb R^n)$, we introduce a quadratic function $g(\cdot;x_0,x_T) : \mathbb R^n \rightarrow \mathbb R$ as follows:
\begin{equation}\label{4:g}
\begin{aligned}
g(\eta_0;x_0,x_T) =\ & \Vert \mathbb K^*\eta_0 \Vert_{L^2_{\mathbb F}(0,T;\mathbb R^{m_0})}^2 +2\langle x_0,\ \eta_0 \rangle -2 \langle x_T,\ \mathbb L^*\eta_0 \rangle_{L^2_{\mathcal F_T}(\Omega;\mathbb R^n)}\\
=\ & \langle \mathbf G \eta_0,\ \eta_0 \rangle +2\langle x_0 -\mathbb L x_T,\ \eta_0 \rangle,
\end{aligned}
\end{equation}
where the relationship \eqref{3.4:K*G} is used in the second equation. Moreover, we propose a family of simple minimization problems.

\medskip

\noindent{\bf Problem (Min).} For any $(x_0,x_T) \in \mathbb R^n \times L^2_{\mathcal F_T}(\Omega;\mathbb R^n)$, to find an $\eta_0^* \in \mathbb R^n$ such that
\begin{equation}\label{4:gMin}
g(\eta_0^*;x_0,x_T) = \inf_{\eta_0 \in \mathbb R^n} g(\eta_0;x_0,x_T).
\end{equation}
The vector $\eta_0^*$ satisfying \eqref{4:gMin} is called a {\it minimizer} of Problem (Min) at the point $(x_0,x_T)$. If the minimizer (uniquely) exists at a point $(x_0,x_T)$, then Problem (Min) is said to be {\it (uniquely) solvable at $(x_0, x_T)$}. If the minimizer (uniquely) exists at every point, then Problem (Min) is said to be {\it (uniquely) solvable}.

\medskip

We give the main result of this section.

\begin{theorem}\label{4:THM}
Let Assumptions (H2) and (H4) hold. Then the following three statements are equivalent:
\begin{enumerate}[(i).]
\item The backward system \eqref{3:BS} is exactly controllable;
\item Problem (Min) is uniquely solvable;
\item Problem (Min) at a point $(x_0^\prime,x_T^\prime)\in \mathbb R^n \times L^2_{\mathcal F_T}(\Omega;\mathbb R^n)$ is unique solvable.
\end{enumerate}
In this case, the unique minimizer of Problem (Min) at $(x_0,x_T) \in \mathbb R^n \times L^2_{\mathcal F_T}(\Omega;\mathbb R^n)$ is
\begin{equation}\label{4:Minimizer}
\eta_0^* = -\mathbf G^{-1} (x_0 -\mathbb Lx_T)
\end{equation}
where $\mathbf G$ is defined by \eqref{3.2:G} and $\mathbb L$ is defined by \eqref{3.4:KL}, and the corresponding minimum is
\begin{equation}\label{4:Minimum}
g(\eta_0^*; x_0,x_T) = - \langle \mathbf G^{-1}(x_0 -\mathbb L x_T),\ x_0 -\mathbb L x_T \rangle.
\end{equation}
Moreover, the following admissible control
\begin{equation}\label{4:vSteer}
v^*(\cdot) := \Big\{ \mathbf B(\cdot)^\top \Phi(\cdot)^\top +\bar{\mathbf B}(\cdot)^\top \mathbb E\big[ \Phi(\cdot)^\top \big] \Big\} \begin{pmatrix} I_n \\ 0_{(Nn) \times n} \end{pmatrix} \eta_0^*
\end{equation}
can steer the state process $x^*(\cdot)$ from $x^*(0) =x_0$ to $x^*(T) =x_T$.
\end{theorem}

\begin{proof}
Firstly, for any $(x_0,x_T) \in \mathbb R^n \times L^2_{\mathcal F_T}(\Omega;\mathbb R^n)$, since $g(\cdot;x_0,x_T)$ is a quadratic function (see \eqref{4:g}), it is clear that Problem (Min) at $(x_0,x_T)$ admits a unique minimizer if and only if
\begin{equation}\label{4T:Eq1}
\mathbf G >0.
\end{equation}
We note that \eqref{4T:Eq1} is independent of $(x_0,x_T)$, then the unique solvability of Problem (Min) at any one point is equivalent to its unique solvability at all points, i.e. the statements (ii) and (iii) are equivalent. Moreover, Theorem \ref{Gram-Cri} reads that \eqref{4T:Eq1} is also equivalent to the exact controllability of the backward system \eqref{3:BS}. Therefore, all the statements (i), (ii) and (iii) are equivalent.

Secondly, by the complete square formula, we derive from \eqref{4:g} that
\begin{equation}
g(\eta_0;x_0,x_T) = \big\langle \mathbf G^{-1} \big( \mathbf G\eta_0 +x_0 -\mathbb Lx_T \big),\ \mathbf G\eta_0 +x_0 -\mathbb Lx_T \big\rangle
- \big\langle \mathbf G^{-1} (x_0 -\mathbb L x_T),\ x_0 -\mathbb L x_T \big\rangle.
\end{equation}
Then we get the conclusions \eqref{4:Minimizer} and \eqref{4:Minimum}.

Thirdly, for any $x_T \in L^2_{\mathcal F_T}(\Omega;\mathbb R^n)$, any $v(\cdot) \in L^2_{\mathbb F}(0,T;\mathbb R^{m_0})$ and any $\eta_0 \in \mathbb R^n$, the dual relationship \eqref{3.4:Dual} implies
\[
\begin{aligned}
0 =\ & \langle \eta_0,\ x^*(0,x_T,v(\cdot)) \rangle -\langle \mathbb L^* \eta_0,\ x_T \rangle_{L^2_{\mathcal F_T}(\Omega;\mathbb R^n)} -\langle \mathbb K^*\eta_0,\ v(\cdot) \rangle_{L^2_{\mathbb F}(0,T;\mathbb R^{m_0})}\\
=\ & \big\langle \eta_0,\ x^*(0,x_T,v(\cdot))-\mathbb L x_T \big\rangle +\mathbb E\int_0^T \Big\langle \mathbf B(s)^\top \xi(s;\eta_0) +\bar{\mathbf B}(s)^\top \mathbb E[\xi(s;\eta_0)],\ v(s) \Big\rangle\, \mathrm ds.
\end{aligned}
\]
By the expression \eqref{3.4:xiPhi}, the above equation is deduced as (the argument $s$ is suppressed for simplicity)
\[
\begin{aligned}
0 =\ & \big\langle \eta_0,\ x^*(0,x_T,v(\cdot))-\mathbb L x_T \big\rangle +\mathbb E\int_0^T \Big\langle \Big\{ \mathbf B^\top \Phi^\top +\bar{\mathbf B}^\top \mathbb E[\Phi^\top]\Big\} \begin{pmatrix} I \\ 0 \end{pmatrix} \eta_0,\ v \Big\rangle\, \mathrm ds\\
=\ & \bigg\langle \eta_0,\ x^*(0,x_T,v(\cdot))-\mathbb L x_T 
+\begin{pmatrix} I & 0 \end{pmatrix} \mathbb E\int_0^T \Big\{ \Phi\mathbf B +\mathbb E[\Phi] \bar{\mathbf B} \Big\} v\, \mathrm ds \bigg\rangle.
\end{aligned}
\]
Now, for any $x_0\in \mathbb R^n$, letting $\eta_0^*$ denote the unique minimizer of Problem (Min) at $(x_0, x_T)$, Equation \eqref{4:Minimizer} provides an expression of $\mathbb L x_T = x_0 +\mathbf G \eta_0^*$. Substituting this expression into the above equation leads to
\[
\begin{aligned}
0 =\ & \bigg\langle \eta_0,\ x^*(0,x_T,v(\cdot)) -x_0 -\mathbf G \eta_0^* 
+\begin{pmatrix} I & 0 \end{pmatrix} \mathbb E\int_0^T \Big\{ \Phi\mathbf B +\mathbb E[\Phi] \bar{\mathbf B} \Big\} v\, \mathrm ds \bigg\rangle\\
=\ & \bigg\langle \eta_0,\ x^*(0,x_T,v(\cdot)) -x_0  
+\begin{pmatrix} I & 0 \end{pmatrix} \mathbb E\int_0^T \Big\{ \Phi\mathbf B +\mathbb E[\Phi] \bar{\mathbf B} \Big\}\\
& \hskip 67mm \bigg[ v -\Big\{ \Phi\mathbf B +\mathbb E[\Phi] \bar{\mathbf B} \Big\}^\top \begin{pmatrix} I \\ 0 \end{pmatrix} \eta_0^* \bigg]\, \mathrm ds \bigg\rangle,
\end{aligned}
\]
where the definition \eqref{3.2:G} of the Gramian matrix $\mathbf G$ is used. By selecting $v(\cdot) =v^*(\cdot)$ defined by \eqref{4:vSteer}, we derive 
\begin{equation}
0 = \big\langle \eta_0,\ x^*(0,x_T,v^*(\cdot)) -x_0 \big\rangle.
\end{equation}
Due to the arbitrariness of $\eta_0$, we have 
\[
x^*(0,x_T,v^*(\cdot)) =x_0,
\] 
i.e., the admissible control $v^*(\cdot)$ given by \eqref{4:vSteer} steers the state process $x^*(\cdot)$ from the initial value $x_0$ to the terminal value $x_T$. The proof is completed.
\end{proof}

\section{Solution of the two-layer problem}\label{SEC:Sol}

In this section, we present a complete solution to the two-layer problem studied in this paper. This solution requires not only Assumptions (H1) and (H4) but also Assumption (H3). We note that the use of Assumption (H3) is to ensure the identity of the following two:
\begin{itemize}
\item the process $\theta(\cdot)$ involved in the exact controllability of MF-GBCS \eqref{2:MF-GBCS} (see Definition \ref{2.2:Def});
\item the solution $\theta(\cdot)$ of MF-FBSDE \eqref{2:MF-FBSDE} used to construct the Nash equilibrium point.
\end{itemize}
Then, the agents' problem at the lower layer and the regulator's problem at the upper layer are linked.

Now, let $x_0 \in \mathbb R^n$ be the initial state and $x_T \in L^2_{\mathcal F_T}(\Omega;\mathbb R^n)$ be the desired terminal state. The two-layer problem can be solved in the following four steps.
\begin{enumerate}[Step 1.]
\item The regulator at the upper layer judges the exact controllability of MF-GBCS \eqref{2:MF-GBCS} by virtue of the Gram-type criterion, or the Kalman-type criterion with the additional Assumption (H5), or the exact observability of the dual system \eqref{3.4:DS} and \eqref{3.4:K*}.

\item When MF-GBCS \eqref{2:MF-GBCS} is exactly controllable, based on the solution $\Phi(\cdot)$ to MF-SDE \eqref{3.2:Found} and the solution $x^*(0;x_T,0)$ to MF-BSDE \eqref{3:BS} and \eqref{3:TC} under zero control, the regulator defines
\begin{equation}\label{5:Sv*}
v^*(\cdot) = -\Big\{ \mathbf B(\cdot)^\top \Phi(\cdot)^\top +\bar{\mathbf B}(\cdot)^\top \mathbb E \big[\Phi(\cdot)^\top \big] \Big\} 
\begin{pmatrix} I_n \\ 0_{(Nn) \times n} \end{pmatrix}
\mathbf G^{-1} \big( x_0 -x^*(0;x_T,0) \big).
\end{equation}

\item Based on the unique solution $(y^*(\cdot)^\top, z^*(\cdot)^\top)^\top = ( (x^{**}(\cdot)^\top, y_{-0}^*(\cdot)^\top ), (q^*(\cdot)^\top, z_{-0}^*(\cdot)^\top ) )^\top$ to MF-BSDE \eqref{3:BS} and \eqref{3:TC} under the admissible control $v^*(\cdot)$ defined by \eqref{5:Sv*}, the regulator further defines and announces the following admissible control (the argument $s$ is suppressed for simplicity):
\begin{equation}\label{5:Su0*}
\begin{aligned}
u_0^* =\ & D_0^\top (D_0D_0^\top)^{-1} \Big\{ \big( q^*-\mathbb E[q^*] \big) -C\big( x^{**} -\mathbb E[x^{**}] \big) -D^y_{-0} \big( y_{-0}^* -\mathbb E[y_{-0}^*] \big)\\
& -D^z_{-0} \big( z_{-0}^* -\mathbb E[z_{-0}^*] \big) \Big\} +\big[ I -D_0^\top (D_0D_0^\top)^{-1}D_0 \big] \big( v^* -\mathbb E[v^*] \big)\\
& +\widehat D_0^\top (\widehat D_0 \widehat D_0^\top)^{-1} \Big\{ \mathbb E[q^*] -\widehat C \mathbb E[x^{**}] -\widehat D^y_{-0} \mathbb E[y_{-0}^*] -\widehat D^z_{-0} \mathbb E[z_{-0}^*] \Big\}\\
& +\big[ I -\widehat D_0^\top (\widehat D_0 \widehat D_0^\top)^{-1}\widehat D_0 \big] \mathbb E[v^*].
\end{aligned}
\end{equation}

\item After the regulator announced her/his admissible control $u_0^*(\cdot)$, the agents at the lower layer play an LQ non-cooperative game at the point $(x_0, u_0^*(\cdot))$ by selecting the following Nash equilibrium point (the argument $s$ is also suppressed):
\begin{equation}
\begin{aligned}
u^{**}_{-0} =\ & -R^{-1} \Big\{ \widetilde B_{-0}^\top \big( y_{-0}^* -\mathbb E[y_{-0}^*] \big) +\widetilde D_{-0}^\top \big(z_{-0}^* -\mathbb E[z_{-0}^*]\big) \Big\}\\
& -\big( R +\bar R \big)^{-1} \Big\{ \big( \widetilde B_{-0} +\bar{\widetilde B}_{-0} \big)^\top \mathbb E[y_{-0}^*] +\big( \widetilde D_{-0} +\bar{\widetilde D}_{-0} \big)^\top \mathbb E[z_{-0}^*] \Big\}.
\end{aligned}
\end{equation}
\end{enumerate}

\end{document}